\documentclass{alggeom}
\pdfoutput=1 
\usepackage[latin1]{inputenc}
\usepackage{multirow}
\usepackage{latexsym,graphics}
\usepackage{amsmath}
\usepackage{subfigure}
\usepackage{amsfonts}
\usepackage{etex}
\usepackage{amsthm}
\usepackage{array}
\usepackage{tikz}	
\usepackage{multirow}
\usepackage{amscd}
\usepackage{amsfonts}
\usepackage{mathrsfs}
\usepackage{amssymb}
\usepackage{amscd, amsfonts}
\usepackage[dvips]{epsfig}
\usepackage{verbatim}
 \usepackage[usenames,dvipsnames]{pstricks}
\usepackage{epsfig}
\usepackage{pst-grad} 
\usepackage{pst-plot} 
\usepackage{pdftricks}
\usepackage{pstricks}
\usepackage{epsf}
\usepackage{multirow}	
\usepackage{amssymb}
\usepackage{verbatim} 
\usepackage{graphicx}
\usepackage{color}
\usepackage{bm}
\usepackage{setspace}
\usepackage{xypic}
\usepackage[all]{xy}
\usepackage{lscape}
\usepackage{hyperref}
\usepackage{enumitem}
\usepackage{tabularx}
\usepackage{multirow}
\usepackage{hhline}
\usepackage{url}
\usepackage{xspace}
\usepackage[numbers,sort&compress]{natbib} 

\numberwithin{equation}{section}
\newtheorem{prop}[equation]{Proposition}
\newtheorem{thm}[equation]{Theorem}
\newtheorem{lemma}[equation]{Lemma}
\newtheorem{cor}[equation]{Corollary}
\newtheorem{defn}[equation]{Definition}

\theoremstyle{definition}
\newtheorem{ex}[equation]{Example}
\newtheorem{rmk}[equation]{Remark}

\newcommand{\p}[1]{\mathbb{P}^{#1}}

\newcommand{\gquot}{/\!\!/}

\newcommand{\SL}{\operatorname{SL}}

\begin{document}
\title{On the GIT Quotient Space of Quintic Surfaces}
\author[Gallardo]{Patricio Gallardo}
\email{pgallardo@math.sunysb.edu}
\address{Department of Mathematics\\Stony Brook University\\ Stony Brook, NY 11794}
%
\begin{abstract} 
We describe the GIT compactification for the moduli space of smooth quintic surfaces in $\p{3}$. In particular, we show that a normal quintic surface with at worst an isolated double point or a minimal elliptic singularity is stable. We also describe the boundary of the GIT quotient,  and we discuss the stability of the non-normal surfaces.
\end{abstract}
\maketitle{}
\section{Introduction}

Horikawa showed that if $X$ is a minimal algebraic surface
with invariants  $p_g=4$, $q=0$ and $c_1^2=5$, and its canonical system $|K_X|$ has no base points, then there exists a birational holomorphic map from $X$ onto a  quintic surface in $\p3$ with  at most ADE singularities (see \cite[ Th. 1]{horikawa}). 
Therefore, our GIT quotient $\overline{\mathcal{M}}^{\scriptsize{GIT}}_5$ is a weakly modular compactification of the loci parametrizing these surfaces.  In a more general context, Gieseker \cite{gieseker} proved the existence of a quasi-projective coarse moduli space for smooth projective surfaces of general type with fixed invariants $p_g$, $q$, and $c_1^2$.  More recently, modular compactifications of theses spaces were constructed by Koll\'ar, Shepherd-Barron  \cite{KSB98}, and Alexeev   \cite{alexeev}.  Currently, we have a limited  understanding of the KSBA compactification $\overline{\mathcal{M}}^{\scriptsize{KSBA}}_5$ 
of the loci parametrizing  quintic surfaces with at most ADE singularities (see \cite{Rana}, \cite{pgthesis}). Therefore, our results are relevant for shedding light in the degenerations of quintic surfaces.

\subsection{Brief Description of the GIT Results}

The GIT quotient space of quintic surfaces (see Figure \ref{fig:GIT}) is a $40$-dimensional projective variety. Next, we give a brief classification of the surfaces parametrized by it.

We characterize stability for normal quintic surfaces.  We show that  quintic surfaces with only isolated double point singularities and  isolated triple point singularities with reduced tangent cone are stable (see Corollary \ref{cor:isolateddp}). 
Normal surfaces whose each singularity has either Milnor number smaller than $22$ or modality smaller than 5 are stable (see Proposition \ref{boundinv}). 
Quintic surfaces with minimal elliptic singularities are stable (Corollary \ref{prop:stableminimalelliptic});  the minimal elliptic surface singularities are analogous to the curve singularities with classical genus drop invariant equal to one  (see Proposition  \ref{effectgenus}).   Surfaces with isolated triple point singularities with non-reduced tangent cone can be both stable and unstable (see Proposition \ref{trpq}). Quintic surfaces with a quadruple point are unstable  (see Proposition \ref{prop:q4pt}).


We give a partial description of stable non-normal quintic surfaces.  We show that quintic surfaces with an irreducible curve of singularities  of genus greater than one are stable  (see Corollary \ref{highergenus}).   A generic quintic surface with a curve of singularities of multiplicity three such that the support of that curve does not contain any line is stable
(see Proposition \ref{mltp3}).  Surfaces with  a triple line are unstable (see Proposition \ref{prop:tripleline}). Quintic surfaces that decompose in a union of a plane and a quartic surface are discussed in Proposition \ref{41}. 

\begin{figure}[h!]
 \centering
\begin{tikzpicture}[scale=0.4]
	\draw[dashed,color=gray](15,4) -- (30,7); \draw (11,5) node {$\dim (\Lambda_2)=1$};
	\draw[dashed,color=gray] (15,16) -- (30,13);
	\draw[very thick] (15,4) -- (15,7);  \filldraw[black] (15,10) circle(1.9pt);
	\draw[dashed,color=gray] (15,7) arc (-90:90: 1 and 1.5); 
\draw[dashed] (20,11.8) circle [x radius=3.5cm, y radius=2.8cm];
\draw[dashed] (20,8.1) circle [x radius=3.5cm, y radius=2.8cm];
\draw[dashed] (25.7,10) circle [x radius=3.5cm, y radius=2.8cm];
\draw (18.5,13) node {Milnor };
\draw (19.9,12) node {number $\leq$ 21};
\draw (19,8.3) node {isolated};
\draw (20,7.3) node {double points};
\draw (26.3,10.5) node {minimal};
\draw (26.3,9.5) node {elliptic sing.};
\draw (26,5) node {stable locus};
\draw (11,10) node {$\dim (\Lambda_{3})=0$}; 
	\draw[dashed,color=gray] (15,10) arc (-90:90: 1 and 1.5); 
\draw (11,14) node {$\dim (\Lambda_{4})=1$};
	
	\draw[very thick] (15,13) -- (15,16);
	\draw[very thick] (30,10) ellipse (0.5 and 3); 
	\draw (34,10) node {$\dim (\Lambda_1)=6$};
	\filldraw[gray] (30,10) ellipse (0.5 and 3);
\end{tikzpicture}
 \caption{GIT quotient space of quintic surfaces}
\label{fig:GIT}
\end{figure}
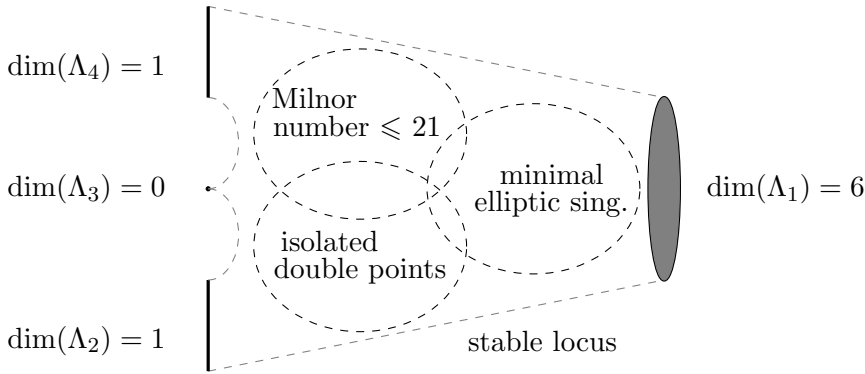
For strictly semi-stable surfaces, we focus only in describing the ones with minimal closed orbits.  We show that the strictly semi-stable locus in the quotient, which we call the GIT boundary, is a union of four disjoint irreducible component $\Lambda_i$
 of dimension $6$, $1$, $0$, and $1$, respectively.
The stabilizer of any GIT semi-stable quintic surface is either equal or contained in a
$\SL(2, \mathbb{C})$ (see Corollary \ref{cor:mxsblz}).  
The generic quintic surface parametrized by the boundary component $\Lambda_1$ is normal, and it has two isolated singularities of multiplicity $3$, geometric genus $3$, modality $7$, and Milnor number $24$.  The generic quintic surface parametrized by the $\Lambda_2$ component is singular along three lines supporting two non-isolated triple point singularities. The  generic quintic surfaces parametrized by the components $\Lambda_{3}$ and $\Lambda_{4}$  are singular along double lines, and they also have  isolated singularities of multiplicity $3$, geometric genus $2$, modality $5$, and Milnor number $24$ and $22$, respectively (see Section \ref{localGIT}).
Strictly semi-stable quintic  surfaces that decompose as a union of a quartic surface and a hyperplane are described in Proposition \ref{41boundary}. 
The only non-reduced  semi-stable quintic surface is a union of a double smooth quadric surface and a hyperplane intersecting along a smooth conic (see Corollary \ref{2qh}). 

\subsection{Organization} 
In Section \ref{sec:GIT}, we present the combinatorial side of the GIT analysis. In particular, we list the critical one-parameter subgroups,  and we give a  description of 
non-stable 
surfaces. In Section \ref{sec:minimalorbit}, we study the strictly semi-stable minimal orbits associated to the GIT compactification.  A more geometric interpretation of the failure of stability in normal surfaces is described in Section \ref{isosing}.  The stability of quintic surfaces with non-isolated singularities is discussed in Section \ref{sec:nonisolatedsing}.   In some cases, we encounter routine computations better done with the help of a computer.  The online notebooks are available at the accompanying website \cite{Website}.

\subsection{Related Work}
This work  fits in a series of GIT constructions including  Shah  \cite{shah4}, 
Laza \cite{laza4}, Yokoyama \cite{yokoyama}, Fedorchuck and Smyth \cite{fedorchuk2013stability}, Lakhani 
\cite{lakhani}
 and Swinarski \cite{swinarski2012git}.  For analyzing the singularities, we benefited from the work of Laufer \cite{laufer},  Prokhorov \cite{prokhorov}, Arnold \cite{arnoldinvetiones}, and others  \cite{m6}, \cite{m5}, \cite{watanabeyoshi}.  Quintic surfaces of general type were also studied by Yang \cite{yang}.  We  used the software Sage \cite{sage}  and Macaulay2 \cite{M2}; in particular, we use the  Macaulay2 package StatePolytope developed by D. Swinarski.  
 
\subsection{Notation}\label{notation}
The homogeneous coordinates are denoted as $[x_0:x_1:x_2:x_3]$.
The homogenous polynomials of degree $d$ are denoted as $f_d(x_0,x_1,x_2,x_3)$.     We work over the complex numbers.  We denote $p_i $  as the point $(x_j=x_k=x_l=0)$ with $i \neq j,k,l$;  we denote $L_{ij}$ as the line$(x_k=x_l=0)$ with $i,j \neq k,l$.
Unless otherwise indicated, whenever a polynomial occurs we suppose it  has generic coefficients. However, it is written without non-zero coefficients.  For example, $c_ix_i^2 + c_kx_k^2$ is written as $x_i^2+x_k^2$.
Furthermore, if we work at the completion of the local ring of a singularity, we do not write the coefficients whenever they are invertible elements. For example, $u(x,y,z)x^2 + v(x,y,z)y^3$ is written only as $x^2+y^3$ if $u(x,y,z)$ and $v(x,y,z)$ are invertible power series.  The equation of $X$ with respect to a given coordinate system is denoted 
as $F_X(x_0,x_1,x_2,x_3)$.  We denote $\Xi_{F_X}$ as its  set of non-zero monomials. 
Similarly, $V(I)$ denotes the zero set of an ideal $I$.  Given a point $ p \in X$, we refers to its projectivized tangent cone  as tangent cone. Our computational framework follows that of Mukai \cite[Sec 7.2]{mukai}.  

\subsection{Acknowledgements} \label{ack} 
I am grateful to my advisor, Radu Laza, for his guidance and help.  I am also grateful to the 
referee for the  detailed comments, corrections, and ideas to simplify some results.  While writing this paper, I have benefited from discussions with V. Alexeev, D. Jensen, J. Rana, J. Tevelev, B. Hassett, A. Sharland, and E. Rosu.   I am also grateful to the CIE and the Mathematics Department at Stony Brook University  for their support. The author was partially supported by  the NSF grant DMS-125481 (PI: R. Laza), and  the W. Burghardt Turner Fellowship.

\section{Geometric Invariant Theory Analysis}
\label{sec:GIT}
Geometric invariant theory provides a standard way to compactify some moduli spaces. 
In particular, the moduli of smooth quintic surfaces is an open  subset of the GIT compactification: 
$$
\overline{\mathcal{M}}^{\scriptscriptstyle{GIT}}_5 = \mathbb{P} \Big( \operatorname{Sym}^5 \big( H^0 \left(\mathbb{P}^3 , \mathcal{O}_{\mathbb{P}^3} (1) \right) \big) \Big)^{ss} 
\gquot
\SL(4, \mathbb{C}).
$$
The stability of a given surface $(F_X=0)$ is determined by using the Hilbert-Mumford numerical criterion.  Let $\lambda$ be a non-trivial  one-parameter subgroup $\lambda \left( t \right) : \mathbb{G}_m \to  \SL \left( 4, \mathbb{C} \right)$.  
The Hilbert-Mumford numerical  function can be 
defined as (for details see \cite[Ch. 9]{dolgachev2003lectures})
\begin{align}\label{po} 
\mu( \lambda, X) & =
\min \lbrace 
\lambda.m_k \; | \; m_k \in \Xi_{F_X}
\rbrace.
\end{align}
A non-trivial 1-PS $\lambda$ is called normalized if it has the form
$$
\lambda = \operatorname{diag} \left( t^{a_0} ,t^{a_1},t^{a_2} , t^{a_3} \right) \text{ with }
a_0 \geq a_1 \geq a_2 \geq a_3 \; \; \text{ and } \; \; a_0 + a_1+a_2 + a_3 = 0.
$$
We assume that our one-parameter subgroups are  normalized. 
This is possible because any 1-PS  is conjugated to a normalized one. 
The Hilbert-Mumford numerical criterion (see \cite[Th. 9.1]{dolgachev2003lectures}) implies that a quintic surface is stable (resp. semi-stable) if and only if for every  normalized $\lambda$ it holds
that $\mu(\lambda, X) < 0$ (resp. $ \leq 0$).

The normalized one-parameter subgroups induce a partial order among the monomials. Indeed, given two monomials $m$, $m'$, then $m \geq m'$ if and only if
$\lambda.m \geq \lambda.m'$ for every normalized 1-PS $\lambda$ (see \cite[Eq 7.11]{mukai}). 
From the definition of the numerical criterion, the minimal monomials in a configuration $\Xi_{F_X}$ are the ones that determine the sign of 
$\mu(\lambda,X)$; and  $X$ is a non-stable surface if and only if there exists a coordinate system and at least one normalized parameter subgroup 
$\lambda= (a_0,a_1,a_2,a_3)$ such that its associated set of monomials $\Xi_{F_X}$ is contained in
$$ 
M^{\oplus}( \lambda ) 
:= 
\left\{
x_0^{i_0 }x_1^{i_1} x_2^{i_2} x_3^{i_3} \; | \; a_0 i_0 + a_1 i_1+a_2 i_2+ a_3 i_3 \geq 0,
\; \; \; 
i_0+i_1+i_2+i_3=5, \; \; i_k \geq 0
\right\}.
$$
For the analysis of stability, it suffices to consider the maximal sets $M^{\oplus}( \lambda ) $ with respect to the inclusion. We call them \emph{maximal non-stable configurations}, and they are determined by a finite list of 1-PS that we call \emph{critical one-parameter subgroups}. 
\begin{prop}
\label{prop:maximalsemistableops}
A quintic surface $X$ is non-stable if and only if for a  choice of a coordinate system its monomial configuration $\Xi_{F_X}$ is contained in 
$ M^{\oplus} \left( \lambda_i \right)$ for one of the following 1-PS:
\begin{align*}
&\lambda_1 = \left( 1, 0, 0,-1 \right) & 
& \lambda_2 = \left( 2, 1, -1, -2 \right) &
&\lambda_3 = \left( 4, 2, -1, -5 \right) & 
\\
&\lambda_4 =\left( 2, 1, 0, -3 \right) &
& \lambda_5= \left( 3, 0, -1, -2 \right) &
& \lambda_6 = \left(5, 1, -2, -4 \right) & \\
& \lambda_{7} = (2,1,1,-4)&
& \lambda_8= \left( 2, 2, -1, -3 \right) &
& \lambda_{9} =(7,1,-4,-4) & 
& \lambda_{10} = (8,-1,-2,-5)
\end{align*}
Furthermore, if for a choice of coordinates $\Xi_{F_X} \subseteq M^{\oplus} \left( \lambda_i \right)$ for $i \geq 7$, then $X$ is unstable. 
\end{prop}
\begin{proof}
Only finitely many configurations of monomials are relevant for the GIT analysis. To find them, with the aid of a computer program (see \cite{Website}), we list all the configurations, and we identify the maximal ones. 
The computation complexity is greatly reduced by using two basic
observations: First, it suffices to consider the configurations associated to 
$M^{\oplus} \left( \lambda \right)$ where $\lambda$ is such that there exist distinct monomials $m_1$, $m_2$ satisfying $\lambda.m_1=\lambda.m_2$. Second, a configuration is characterized by its set of minimal monomials with respect to the previously defined partial order.   We also ensure our list of critical 1-PS is complete. Indeed, by examining the equation
$\lambda.m_1=\lambda.m_2$ with $\lambda=(a_0,a_1,a_2,a_3)$   is clear that $|a_i|<3(5)^3$ with $a_i \in \mathbb{Z}$. By using criterion \cite[Prop. 7.19]{mukai}, we confirm that $M^{\oplus}(\lambda) \subset M^{\oplus}(\lambda_k)$ for every $\lambda$ of that form.  Our implementation of the algorithm to find the maximal sets $M^{\oplus}(\lambda_k)$  and to ensure our list of critical 1-PS is complete follows similar cases in the literature (e.g. \cite{laza4},\cite{lakhani}).   Finally, the generic configuration associated to each $M^{\oplus}(\lambda_k)$ can be either semi-stable or unstable depending on the presence or absence of the point $\left( \frac{5}{4}, \frac{5}{4}, \frac{5}{4},\frac{5}{4} \right)$ in the convex hull spanned by the monomials in $M^{\oplus}(\lambda_k)$. We distinguish the semi-stable configurations from the unstable  configurations by using the Macaulay2 package StatePolytope developed by D. Swinarski.
\end{proof}

For each $k = 1, \ldots, 10$, our goal is to describe the geometric properties of surfaces $X$ such that $\Xi_{F_X} \subset M^{\oplus} \left( \lambda_k \right)$. 
Our first step is to recall  that each normalized $\lambda$ acts on the vector space 
$W:=H^0(\mathbb{P}^3, \mathcal{O}_{\mathbb{P}^3}(1))$, determining a weight decomposition $W =\oplus_{s}W_s$. This decomposition induces a (partial) flag of subspaces $(F_n)_m := \oplus_{s \leq m} W_s \subset W$ that  determines a (partial) flag 
$
(F_n)_{\lambda}:=p_{\lambda} \subset L_{\lambda} \subset H_{\lambda} \subset \p{3}$. 
For instance, in our coordinate system if the normalized $\lambda$ has different weights $a_i$, the flag $(F_n)_{\lambda}$ is:
$$
\big( p_{\lambda}:= [0:0:0:1] \big)
\in 
\big( L_{\lambda}:=V(x_0, x_1) \big) \subset 
\big( H_{\lambda}:=V(x_0) \big).
$$
We say that $(F_n)_{\lambda}$ is a \emph{bad flag for the surface $X$ with respect to $\lambda$}, if  $\mu(\lambda, X) \geq 0$. Next, we describe the singularities of $X$ singled out by their bad flag with respect to  $\lambda_k$.

\begin{table}[h]
\centering
\caption{ 
Singularities  of a generic 
surface $X$ such that $\Xi_{F_X} \subseteq M^{\oplus} \left( \lambda_k \right)$
(Propositions \ref{prop:GIT} and \ref{prop:unstableGIT})
}
\begin{tabular}{ | p{2cm}| p{11cm} | }
\hline
1-PS& Associated Geometric Characteristics \\ \hline
$\lambda_1$, $\lambda_3$, $\lambda_4$ 
& 
Isolated triple point singularity with non-reduced tangent cone
\\ \hline
$\lambda_7$& 
Isolated ordinary quadruple point singularity
\\ 
\hline
$\lambda_2$, $\lambda_6$, $\lambda_8$
& 
Double line of singularities supporting a non-isolated triple point
\\ \hline
$\lambda_5$, $\lambda_9$& 
Double line of singularities with a distinguished double point
\\ 
\hline
$\lambda_{10}$& 
A union of a quartic surface and a hyperplane
\\
\hline
\end{tabular}\label{cr1ps}
\end{table}

\begin{prop}\label{prop:GIT}
Let $X$ be a quintic surface, and let $\Delta$ be its singular locus. If $X$ is a strictly semi-stable quintic surface with isolated singularities, then
\begin{enumerate}[label*=\arabic*.]
\item 
$\Delta$ contains a triple point singularity $p \in X$ whose tangent cone is a union of a double plane $H^2$ and another plane. The intersection multiplicity of the surface with any line in $H$ containing the triple point is five.
\item 
$\Delta$ contains a triple point singularity $p \in X$ whose tangent cone is a union of a double plane $H^2$ and another plane intersecting $H$ along a line $L$ which is contained in $X$. The intersection of the hyperplane $H$ with the surface $X$ is a union of a double line $L^2$ and a nodal cubic plane curve such that the double line is tangent to the cubic curve at the node.
\item 
$\Delta$ contains a triple point singularity $p \in X$ whose tangent cone is a triple plane $H^3$. The quintic plane curve obtained from the intersection of the surface $X$ with $H$ has a quadruple point whose tangent cone contains a triple line.
\end{enumerate}
If $X$ is an irreducible strictly semi-stable quintic surface with non-isolated singularities, then 
\begin{enumerate}[label*=\arabic*.] \setcounter{enumi}{3}
\item 
$\Delta$ contains a double line $L^2$ supporting a special double point whose tangent cone is $H^2$. At the completion of the local ring, the equation associated to the double point has the form (see Notation \ref{notation})
$$
x^2 + y^2z^2f_2(y,z^2)+y^5.
$$
The intersection of $X$ with $H$ is a quintuple line supported on $L$. 
\item 
$\Delta$ contains a double line $L^2$ supporting a special triple point $p \in X$. 
The tangent cone of the triple point is a union of three planes intersecting along $L$. At the completion of the local ring, the equation associated to the triple point has the 
following form
$$
xf_2(x,y)+y^3z+y^4+x^2z^2+xyz^3
$$
The intersection of the surface with one of the above hyperplanes $H$  is a union of a conic and a transversal triple line supported on $L$. 
\item 
$\Delta$ contains a double line $L^2$ supporting a special triple point whose tangent cone is a union of a double plane $H^2$ and another plane. 
At the completion of the local ring, the equation associated to the triple point has the following form
$$
x^2y+x^4+y^4+x^3z+x^2z^2+xy^3+xy^2z+xyz^3.
$$
The intersection of the surface with the hyperplane $H$ is a union of a quadruple line supported on $L$ and another line.
\end{enumerate}
\end{prop}
\begin{proof}
We suppose the quintic surface is strictly semi-stable. By Proposition \ref{prop:maximalsemistableops}, we only need to find the geometric characterization 
of the quintics defined by  the equations $F_{\lambda_i}$ for 
$ 1 \leq i \leq 6$.  The statement describes 
the intersection of these surfaces with the corresponding flag  $(F_n)_{\lambda_i}$.
We order the cases as in the statement.  
The equation associated to the first case is
$$
F_{\lambda_1}= x_3^2 x_0^2 f_1(x_0,x_1,x_2)+ x_3 x_0 f_3(x_0,x_1,x_2)+ f_5(x_0,x_1,x_2).
$$
The equation associated to the second case is
\begin{align*}
F_{\lambda_3}= 
x_3^2 x_0^2 f_1(x_0 , x_1 ) +
x_3(x_0^2x_2^2+x_2f_3(x_0,x_1)+f_4(x_0,x_1) )
+ x_1^2 f_3 (x_1 , x_2 ) + x_0 f_4(x_0,x_1,x_2).
\end{align*}
The equation associated to the third case is
\begin{align*}
F_{\lambda_4} &= 
x_3^2 x_0^3 + x_3 x_1^3f_1(x_1 , x_2 ) + x_3x_0^3 h_1(x_0,x_1,x_2)
+ x_3 x_0^2 f_2 (x_1 , x_2 ) 
\\
& \qquad{} + x_3 x_0 x_1 g_2 (x_1 , x_2 ) + f_5(x_0,x_1,x_2).
\end{align*}
The equation  associated to the fourth case is
\begin{align*}
F_{\lambda_5} &=
x_3^3 x_0^2 + x_3^2 x_0^2 f_1(x_0,x_1,x_2) + 
x_3x_0^2 f_2(x_0,x_1,x_2)
+ x_3 x_0 x_1^2 f_1 (x_1 , x_2 )
+ x_0^2 f_3(x_0,x_1,x_2)
\\
& \qquad
+ x_0 x_1 g_3 (x_1 , x_2 ) + a_1 x_1^5.
\end{align*}
The equation associated to the fifth case is
\begin{align*}
F_{\lambda_2}
& =
x_3^2 x_0f_2(x_0 , x_1 ) + x_3x_1^3 f_1(x_1 , x_2 ) + x_3x_0x_1^2h_1(x_1 , x_2 )
+x_3x_0^2 g_2(x_0,x_1,x_2)
\\
& \qquad
+ x_0^2f_3(x_0,x_1,x_2) +
x_0 x_1g_3 ( x_1 , x_2 ) + x_1^3 h_2 (x_1 , x_2 ).
\end{align*} 
The equation associated to the last case is
\begin{align*}
F_{\lambda_6}
&=
x_3^2 x_0^2 f_1(x_0,x_1,x_2)+ x_3x_0^2 f_2(x_0,x_1,x_2)
+ x_3x_0 x_1^2 h_1(x_1 , x_2 )+ x_3x_1^4 
\\
& \qquad
+ x_1^4 g_1(x_1 , x_2 ) + x_0 x_1 f_3 (x_1 , x_2 ) + x_0^2 g_3(x_0,x_1,x_2).
\end{align*}
To find the local equations of the singularities we use an analytic change of coordinates as described by \cite[Sec 2.5]{kollar1998real}.
\end{proof}
Next we describe the main geometric characteristics of the surfaces destabilized by the critical 1-PS $\lambda_7$, $\lambda_8$, $\lambda_9$, and $\lambda_{10}$.
\begin{prop} \label{prop:unstableGIT}
Let $X$ be a quintic surface, and let $\Delta$ be its singular locus. 
Suppose that for some coordinate system $\Xi_{F_X} \subset M^{\oplus}(\lambda_k)$ with 
$k \geq 7$. Then $X$ is an unstable quintic surface, and one of the following cases holds:
\begin{enumerate}[label*=\arabic*.]
\item 
$\Delta$ contains an ordinary quadruple point.
\item 
$\Delta$ contains a double line supporting a special triple point $p \in X$ whose tangent cone is a union of three concurrent hyperplanes intersecting along a line $L$. At the completion of the local ring, the equation associated to the triple point has the following form
$$
f_3(x,y)+y^2z^3+xyz^3+x^2z^3.
$$
The intersection of the surface with one of the above hyperplanes is a union of a cubic curve and a tangent double line supported at $L$.
\item 
$\Delta$ contains a double line supporting a special double point whose tangent cone is $H^2$. At the completion of the local ring, the equation associated to the double point has the form 
$x^2+y^4$.  
The intersection of the surface with $H$ is a union of a quadruple line supported on $L$ and another line.
\item 
$X$ is a union of a smooth quartic surface and a hyperplane such that
\begin{itemize}
\item[*] The intersection of the hyperplane with the quartic surface is a quartic plane curve with a triple point whose tangent cone has a triple line $L^3$. 
\item[*] The intersection of the quartic surface with this line $L$ is a quadruple point. 
\end{itemize}
\end{enumerate}
\end{prop}
\begin{proof} 
We write down the equations of surfaces destabilized by the critical 1-PS $\lambda_i$ with 
$i \geq 7$ and match them with the cases of the statement.
The generic equation associated to the first case is given by
\begin{align*}
F_{\lambda_7} = x_3f_4(x_0,x_1,x_2)+f_5(x_0,x_1,x_2,x_3).
\end{align*}
The equation associated to the second case is given by
\begin{align*}
F_{\lambda_8}&= x_3^2f_3(x_0,x_1) +x_3 (x_2f_3(x_0,x_1)+f_4(x_0,x_1))
+x_2^3f_2(x_0,x_1) +x_2f_4(x_0,x_1) +f_5(x_0,x_1).
\end{align*}
The equation associated to the third case is given by
\begin{align*}
F_{\lambda_9}&=
x_3^3 x_0^2 + x_3^2 x_0^2 f_1(x_0,x_1,x_2) + x_3^2 x_0x_1^2
+ x_3x_2^2x_0^2+x_3x_2x_0f_2(x_0,x_1)
\\
& \quad
+x_3f_4(x_0,x_1)
+ x_2^3x_0^2 +x_2^2x_0f_2(x_0,x_1) + x_2f_4 (x_0 , x_1 ) + f_5(x_0,x_1).
\end{align*}
The equation associated to the fourth case is given by
\begin{align*}
F_{\lambda_{10}} &= 
x_0 \left( 
x_3^3x_0+x_3^2x_0g_1(x_0,x_1,x_2)
+x_3x_2x_0f_1(x_0,x_1,x_2)+x_3f_3(x_0,x_1) 
+f_4(x_1,x_2,x_0) \right).
\end{align*}
\end{proof}
Next, we discuss some additional stability results.
\begin{cor}\label{convL1}
Let $X$ be a normal quintic surface with a triple point whose tangent cone is non-reduced. Let $\tilde{X} \to X$ be the monomial transformation of $X$ with its center at the triple point.  Then $\tilde{X}$ is non-normal if and only if there is a coordinate system such that $\Xi_{F_X} \subset \Xi_{F_{\lambda_1}}$. 
\end{cor}
\begin{proof}
We can select a coordinate system such that the triple point is supported at $p_3$ and its tangent cone  is $(x^2_0f_1(x_0,x_1,x_2)=0)$. 
In that case, the equation of the quintic surface can be written as:
$$
x_3^2 x^2_0f_1(x_0,x_1,x_2)+x_3f_4(x_0,x_1,x_2)+f_5(x_0,x_1,x_2).
$$
The singularities of $\tilde{X}$  are supported on the exceptional divisor. 
As found by \cite[Prop. 4.2]{yang}, this happens along the intersection of $(x_0=0)$ with $(f_4(x_0,x_1,x_2)=0)$, and the failure of normality of $\tilde{X}$ is equivalent to $x_0 \mid f_4(x_0,x_1,x_2)$.  
Therefore, $\tilde X$ is non-normal if and only if the equation of $X$ can be written as
$$
x_3^2 x^2_0f_1(x_0,x_1,x_2)+x_3x_0f_3(x_0,x_1,x_2)+f_5(x_0,x_1,x_2).
$$
The statement follows by inspecting the equation $F_{\lambda_1}$ in the proof of Proposition \ref{prop:GIT}.
\end{proof}
\begin{cor}
\label{cor:isolateddp} 
Let $X$ be a normal quintic surface 
such that each of its singularities is either an isolated double point  or
an isolated triple point  whose tangent cone is reduced.  Then $X$ is stable. 
\end{cor}
\begin{proof}
It follows from Propositions \ref{prop:GIT} and \ref{prop:unstableGIT} because  the singularities of non-stable normal quintic surfaces are necessarily worse than
triple points with non-reduced tangent cone. 
\end{proof}
\begin{prop} \label{prop:tripleline}
Let $X$ be a quintic surface containing a line $L$ of singularities such that 
$mult_p(X) =3$ for all $p \in L$.  Then $X$ is unstable.
\end{prop}
\begin{proof}
We can suppose the triple line is supported at $(x_0=x_1=0)$.  Then the equation of $X$ can be written as
$$
g_2(x_0,x_1,x_2,x_3)x_0^3+f_2(x_2,x_3)x_0^2x_1+h_2(x_0,x_1,x_2,x_3)x_0x_1^2
+ p_2(x_1,x_2,x_3)x_1^3
$$
which is destabilized by $\lambda_{8}$ (see \cite{Website}). 
\end{proof}
\begin{prop}\label{prop:q4pt}
Let $X$ be a quintic surface with a singularity of multiplicity greater than or equal to four. Then $X$ is unstable.
\end{prop}
\begin{proof}
Suppose the quadruple point is supported at $p_3$. Then $X$ is destabilized by $ \lambda_7$.
\end{proof}

\begin{lemma}
\label{triplepointgenus}
Let $X$ be an irreducible quintic surface such that  its singular locus $\Delta$ 
 contains a  non-planar reduced curve $C$.  Then  
$\deg(C) \leq 6$.  Furthermore, if $X$ has at least one triple point singularity, then $C$ is either a twisted cubic, an elliptic quartic curve, or a degeneration of these. 
\end{lemma}
\begin{proof}
We apply the genus  formula  to the generic section of $X$ which is an irreducible plane quintic curve that cannot  have more than 6 double points.  Those double points are induced by $C$. Then, if $X$ is irreducible  the degree of $C$ is less than $6$.  The same argument applies if $X$ has a triple point. We take a section of $X$ through it, and we use that 
there is not a curve  of degree four and 
genus two  in $\mathbb{P}^3$.
\end{proof}

\begin{cor} \label{highergenus} 
Let $X$ be an irreducible quintic surface with a curve of singularities supported on a reduced curve $C$. Suppose the genus of $C$ is greater than one,  $C$ does not contain any line,
and $X$ does not have an additional line of singularities.
Then $X$ is stable. 
\end{cor}
\begin{proof}
Lemma \ref{triplepointgenus} and our hypothesis
about the genus of $C$ imply that $X$ has no triple point singularities.  Then $X$ is either stable or $\Xi_{F_X} \subset M^{\oplus}(\lambda_i)$ for $i \in \{5, 9 \}$  (see Table \ref{cr1ps}). However, this last case implies that $X$ contains a line of singularities.
\end{proof}

To decide the semi-stability of a quintic surface with a 
$ \operatorname{SL}(2, \mathbb{C})$-stabilizer, we make use of its symmetry to reduce the number of 1-PS for which we have to check the  Hilbert-Mumford numerical criterion (for a similar argument see \cite[Prop. 2.4]{alper2013finite}).
\begin{lemma}\label{sl2c}
Let $X$ be a quintic surface that decomposes in a union of a quartic surface and a hyperplane. Suppose there is a $\operatorname{SL}(2,\mathbb{C}) \subset Aut(X)$ action that fixes a smooth conic, $C$, on $X$. Then, there is a coordinate system $\{ x_i \}$ such that the equation associated to $X$ has the form
$$
x_1\big(
f_2(x_0,x_2,x_3)^2 + x_1f_3(x_0,x_1,x_2,x_3) 
\big)
$$
where $(x_1=f_2(x_0,x_2,x_3)=0)$ defines the invariant conic. Moreover, the quintic surface $X$ is semi-stable if and only if it is semi-stable with respect to every 1-PS acting diagonally on $\{ x_i \}$ and of the form 
$\lambda = \operatorname{diag} (a_0,a_1,a_2,a_3)$ with 
$a_0 \geq a_2 \geq a_3$.
\end{lemma}
\begin{proof}

A basis $\{ x_i \}$ of a vector space $W$ is compatible with a reductive group  if given an equivariant decomposition of $W$, the equivariant subspaces are spanned by a subset of the variables $\{x_i \}$.  For us,  the group is $\operatorname{SL}(2,\mathbb{C})$, $W:=H^0(\mathbb{P}^3, \mathcal{O}_{\mathbb{P}^3}(1))$ and  $V \cong H^0(\mathbb{P}^1, \mathcal{O}_{\mathbb{P}^1}(1))$ is the standard two dimensional 
$\SL(2, \mathbb C)$-representation.   We select a distinguished coordinate system $\{x_i \}$ compatible with the $\SL(2, \mathbb C)$-decomposition
\begin{align}\label{dpW}
W:=H^0(\mathbb{P}^3, \mathcal{O}_{\mathbb{P}^3}(1)) \cong \operatorname{Sym}^2(V) \oplus \operatorname{Sym}^0(V)
\end{align}
induced by the embedding $C \hookrightarrow \mathbb{P}^3$. 
In particular, $\mathbb{P}(\operatorname{Sym}^2(V)):= (x_1=0)$ is the plane containing  $C$.  We select a maximal torus $T_{max}$ compatible with the $\{ x_i \}$. It follows that the plane $(x_1=0)$ is fixed by $T_{max}$,  and the equation of $X$ in this coordinate system is the one of the statement.  
We follow the notation of \cite[Def 4.6]{morrison2011grobner} by saying that 
 $T_{max}$ determines stability for $X$ because $W$ has a multiplicity-free decomposition  into irreducible $\SL(2, \mathbb C)$-representations and the basis $\{ x_i \}$ is compatible with the $\SL(2, \mathbb C)$-action.
According to  \cite[Prop 4.7 ]{morrison2011grobner}, if $X$ is $T_{max}$-semi-stable, then $X$ is
$\SL(4, \mathbb C)$-semi-stable.  Our result follows because we can take $\lambda \subset T_{max}$ to be $\lambda = diag (a_0,a_1,a_2,a_3)$ with $a_0 \geq a_2 \geq a_3$, where this last condition is achieved by relabelling. 
\end{proof}
\begin{prop}\label{q2h}
Let $X$ be a quintic surface that decomposes as a double quadric surface $Q$ and a hyperplane $H$. Then $X$ is semi-stable if and only if $Q$ is smooth and $Q \cap H$ is smooth.
\end{prop}
\begin{proof} 
The quadric surface must be smooth; otherwise $X$ will contain a quadruple point.  If the intersection $Q \cap H$ is singular,  then the equation associated to the quintic surface can be written as
$
x_1 \big( x_0^2+x_0x_2+x_1^2+x_1x_3 \big)^2
$
because there is exactly one orbit of such quintics.  This last quintic surface  is destabilized by $\lambda_8$.
Next, we suppose that the conic $Q \cap H$ is smooth. 
Then in some coordinate system 
$
F_X=x_1(x_1x_2-x_0x_3+\alpha x_2^2)^2
$
where $\alpha \neq 0$. The semi-stability follows from Lemma \ref{sl2c} and by noting that
$F_X$ is clearly semi-stable with respect to every 1-PS acting diagonally on the $x_i's$.
\end{proof}
\begin{cor}\label{2qh}
A non-reduced quintic surface $X$ is semi-stable if and only if 
$X=2Q+H$ where $Q$ is a smooth quadric surface, and $H$ is a hyperplane intersecting $Q$ along a smooth conic.
\end{cor}
\begin{proof}If $X$ decomposes as a union of a double plane and another cubic surface, then we can select a coordinate system so that $F_X=x_0^2p_3(x_0,x_1,x_2,x_3)$, which 
is destabilized by $\lambda_{10}$. By degree considerations, the other case is a union of a double quadric surface $Q^2$ and a hyperplane $H$. 
Then, the statement follows from Proposition \ref{q2h}.
\end{proof}

\section{
Minimal Orbits of the GIT Compactification}
\label{sec:minimalorbit}

Recall that we refer to the image of  the strictly semi-stable locus in the GIT quotient
as the GIT boundary.  Given a point $q$  at the GIT boundary, there is a unique closed orbit associated to $q$.  If we say that $X$ is parametrized by $q$, then we suppose that $X$ corresponds to that closed orbit.  Next, we describe the generic surfaces parametrized by the GIT boundary and  several aspect of its geometry.
\begin{thm}
\label{gitboundary}
The strictly semi-stable locus in  $\overline{\mathcal{M}}^{\scriptsize{GIT}}_5$ has four disjoint irreducible components:  $\Lambda_1$, $\Lambda_2$, $\Lambda_{3}$ and 
$\Lambda_{4}$ of dimensions $ 6$, $1$, $0$, and $1$, respectively (see Figure \ref{fig:GIT}).
Let $X_k$ be a generic  surface parametrized by $\Lambda_k$. Then $X_k$ has the following geometric properties:
\begin{enumerate}[label*=\arabic*.]
\item The surface $X_1$ is normal; it contains two isolated triple point singularities 
 which are called $V^*_{24}$ on \cite[Table II]{m6}. This singularity has  geometric genus $3$, modality $7$, and Milnor number $24$. 
\item The surface $X_2$ is singular along three lines that support two non-isolated triple point singularities. 

\item The surface $X_3$ has a triple point isolated singularity of geometric genus $2$, modality $5$, and Milnor number $24$, which is called $V_{24}^{*1}$ in \cite[pg 244]{m5}. 

Additionally, the surface $X_3$ is singular along a line supporting a distinguished triple point singularity.

\item 
The surface $X_4$ has an isolated triple point singularity of geometric genus $2$, modality $5$, and Milnor number $22$, which is called $V^{'}_{22}$ in \cite[pg 244]{m5}. 

Additionally, the surface $X_4$ is singular along a line supporting only singularities of multiplicity two.
\end{enumerate}
The boundary component $\Lambda_1$ is associated to the 1-PS $\lambda_1$, $\Lambda_2$ is associated to $\lambda_2$, 
$\Lambda_{3}$ is associated to $\lambda_3$ and 
$\lambda_6$, and $\Lambda_{4}$ is associated to $\lambda_4$ and $\lambda_5$.
\end{thm}
\begin{proof} 
The statements about $X_k$ follow from studying the equations associated to the monomial invariants with respect to $\lambda_k$ for $k \leq 6$ and comparing with the references 
\cite{m5}, \cite{m6} and \cite{watanabeyoshi}. 
We also need some stability calculations. The main theoretical tool is the centralizer version of the Luna's criterion  (see \cite[Remark pg 237]{luna}).
Luna's result implies that given $ W =H^0( \p{3}, \mathcal{O}_{ \p{3}}(5))$, the group  $G= \operatorname{SL}(4, \mathbb{C})$, and the stabilizer $G_x$ of $x$, the orbit  $G \cdot x$ is closed in $W$ if and only if the orbit $C_G (G_x) \cdot x$ is closed in  $W^{G_x}$, where $W^{G_x} \subset W$ denotes the invariant set under the $G_x$ action and $C_G(G_x)$ is the centralizer of $G_x$ in $G$. 
Next, we complete the analysis of the boundary associated to each $\lambda_k$.

\textbf{Boundary stratum $\Lambda_1$.}
The equation of $X_1$ can be written as
\begin{align}\label{mo1}
\overline{F}_{\lambda_1} &=
 x_3^2 x_0^2 f_1(x_1,x_2) + x_3 x_0 f_3 (x_1 , x_2 ) + f_5 (x_1 , x_2 ). 
\end{align}
The generic quintic surface defined by Equation \eqref{mo1} has  two triple point singularities at the points  $[1:0:0:0]$  and $[0:0:0:1]$
of the type described in the statement.   The centralizer 
$C_G(\lambda_1)$  is given by
\begin{align*}
C_G(\lambda_1)=
\left\{  
\begin{pmatrix}
a & 0& 0 \\
0 & A & 0 \\
0 & 0 & \frac{1}{a \det A}
\end{pmatrix}
 : A \in GL(2, \mathbb{C}), \; a \neq 0  \right\}
\end{align*}
The dimension of $\Lambda_1$ is  computed by noting that 
$\dim(W^{\lambda_1})=12$, the centralizer has dimension $5$, and $\overline{F}_{\lambda_1}$ can be defined up to a constant.   From the unstable degenerations of $X_1$, we only describe the one  used to prove that the boundary locus $\Lambda_1$ is disjoint from the other loci $\Lambda_i$.

\textbf{Claim A}: There is not a semi-stable quintic surface $X$ such that
\begin{enumerate}
\item $X$ is parametrized by a point in $\Lambda_1$. 
\item $X$ has at least two triple point singularities $q_A$ and $q_B$ which are degenerations of the triple point singularities in $X_1$.
\item  
X has a line of singularities supported  at the line spanned by the points $\{ q_A, q_B \}$.
\end{enumerate}
\textbf{Proof:}
By hypotheses $(i)$ and $(ii)$, there is a coordinate system such that the equation of $X$ has the form of Equation \eqref{mo1} with the triple points $q_A$ and $q_B$ supported at 
$[1:0:0:0]$ and $[0:0:0:1]$. The line spanned by those points is supported at $(x_1=x_2=0)$.
Hypothesis $(iii)$ implies that the equation of $X$ is a degeneration of the Equation  \eqref{mo1} with the term $f_1(x_1,x_2)$ equal to zero.  The vanishing of that term implies that $X$ has a triple line of singularities; and a  surface with those singularities is unstable by Proposition \ref{prop:tripleline}.

\textbf{Boundary stratum $\Lambda_2$.}
The centralizer of $\lambda_2$ is the torus and the stability analysis is the standard one.  The equation of a  semi-stable surface  parametrized by $\lambda_2$ is given by:
\begin{align}
\label{mo2}
\overline{F}_{\lambda_2} &=
x_3^2 x_0 x_1^2 + x_3 x_0^2 x_2^2 +a_1 x_3 x_2 x_1^3 + a_2 x_0 x_1 x_2^3
& & \text{ where $[a_1:a_2] \in \mathbb{P}^1$.} 
\end{align}
This surface is singular along the lines
\begin{align*}
(x_2=x_3=0)
& &
(x_1=x_2=0)
&& 
(x_0=x_1=0)
\end{align*}
By analyzing Equation \eqref{mo2} and its partial derivatives,  it follows that if 
$a_1$ and $a_2$ are not equal to  $0$, then the surface defined by $\overline{F}_{\lambda_2}$ has only two non-isolated triple point singularities at $[1:0:0:0]$ and $[0:0:0:1]$.   The points $[0:1]$ and $[1:0]$ in $\Lambda_2$ parametrize a union of a quartic surface and a hyperplane which is described in Proposition \ref{41boundary}.

\textbf{Boundary stratum $\Lambda_3$.}  The equation of a semi-stable surface 
stabilized by  $\lambda_3$  is given by:
\begin{align}
\label{mo3}
\overline{F}_{\lambda_3} &=
x_3^2 x_0^2 x_1 + x_3 x_1^3 x_2 + x_0 x_2^4. 
\end{align}
This surface  has an isolated triple point supported at $[0:0:0:1]$
and a line of singularities supported at  $(x_2=x_3=0)$.  The line of singularities supports a triple point at $[1:0:0:0]$.

\textbf{Boundary stratum $\Lambda_4$.}  The equation of a semi-stable surface stabilized by $\lambda_4$ is given by:
\begin{align}
\label{mo4}
\overline{F}_{\lambda_4} &= 
x_0^3x_3^2 + x_3 x_2 x_1^3 + a_1 x_0 x_1 x_2^2x_3 + a_2 x_2^5
& & \text{ where $[a_1:a_2] \in \mathbb{P}^1$}.
\end{align}
If  $a_2 \neq 0$, then the point $[a_1:1] \in \Lambda_4$ parametrizes a quintic surface with one isolated triple point  at  $[0:0:0:1]$ and a line of singularities of multiplicity two supported at $(x_2=x_3=0)$.  The point $[1:0] \in \Lambda_4$ parametrizes a 
union of a quartic surface and a hyperplane described in Proposition \ref{41boundary}.

Finally, we observe that the equations associated to the monomial invariants with respect to 
$\lambda_5$ and $\lambda_6$ are, after a change of coordinates, equal to the Equations
\eqref{mo4} and \eqref{mo3}, respectively.  

Next, we show that the boundary components $\Lambda_i$ are disjoint from each other.

\textbf{Claim:}  $\Lambda_1 \cap \Lambda_2 =\emptyset$. 

\textbf{Proof:} 
From our previous discussion,  if $X$ is parametrized by $\Lambda_2$,  then 
$X$ has either  two triple point singularities with a curve of singularities supported at the line spanned by them, or  decomposes as a union of a quartic and a hyperplane.  
The former case cannot be a degeneration of $X_1$ by  Claim A.
If $X$ decomposes as a union of a quartic and a hyperplane, then  its equation can be written as
\begin{align}
x_3
\left(
x_3 x_0 x_1^2 + x_0^2 x_2^2 + x_2 x_1^3 
\right).
\end{align}
The surface $X$  has three triple point singularities supported at the points
\begin{align*}
q_0=[1:0:0:0]
& &
q_3=[0:0:0:1]
& &
q_2=[0:0:1:0]
\end{align*}
such that the following holds:
\begin{enumerate}
\item $X$ has a two lines of singularities supported at the lines spanned by
$\{q_0, q_3\}$ and by $\{q_2, q_3 \}$.  
\item $X$ does not contain the line 
$(x_1=x_3)$
spanned by 
$\{ q_0, q_2 \}$.
\end{enumerate}
We claim that $X$ cannot be a degeneration of $X_1$. Indeed,  the two triple point singularities in $X_1$ must degenerate to  triple points in $X$, and the line spanned by the two triple points in $X_1$ must be contained in  $X$ as well. 
However, conditions $(i)$ and $(ii)$ imply that $X$  is singular along that line which is  impossible by Claim A.

\textbf{Claim}: $\Lambda_1 \cap \Lambda_3 = \emptyset$.

\textbf{Proof:} 
If $X_3$ is a degeneration of $X_1$, then an isolated triple point singularity of $X_3$ is  a degeneration of one in $X_1$ because $X_3$ has only two triple point singularities.
  This degeneration is impossible by the  semi-continuity of the geometric genus (see \cite[Th. 1]{elkik1981rationalite}).

\textbf{Claim}: $\Lambda_1 \cap \Lambda_4 = \emptyset$.

\textbf{Proof:} 
A surface $X$ parametrized by $\Lambda_4$ has either one  isolated triple point singularity or is a union of a quartic surface and a hyperplane. 
We can rule out the first case by using the 
semi-continuity of the geometric genus as in the case $\Lambda_1 \cap \Lambda_3$. 
If 
$X$ decomposes as a union of a quartic and a hyperplane, then its equation can be written as
$$
x_3(x_0^3x_3+x_2x_1^3+x_0x_1x_2^2).
$$
It follows that $X$ has only two triple point
 singularities supported at 
\begin{align*}
q_2=[0:0:1:0]  && q_3=[0:0:0:1]
\end{align*}
and a line of singularities supported at  $(x_0=x_1=0)$ which is spanned by the points $\{ q_2, q_3 \}$.  Therefore,  $X$ cannot be a degeneration of $X_1$ by Claim A.

\textbf{Claim}: $\Lambda_3 \cap \Lambda_2 =  \Lambda_3 \cap \Lambda_4 =\emptyset$.

\textbf{Proof:} 
A surface $X$ parametrized by $\Lambda_2$  has three lines of singularities while $X_3$ only has one.  Then $X_3$ cannot be a degeneration of $X$.  The surface $X_4$  has an isolated triple point singularity whose tangent cone is of the form  $x^3$. The surface $X_3$ has an isolated triple point singularity whose tangent cone is of the form $x^2y$. Then $X_3$ is not a degeneration of $X_4$.

\textbf{Claim}: $\Lambda_4 \cap \Lambda_2 =\emptyset$.

\textbf{Proof:} 
A surface parametrized by $\Lambda_2$ does not have an isolated triple point singularity.
By our previous discussion,  $\Lambda_4 \cap \Lambda_2$ must parametrize a union of a quartic surface and a hyperplane.   
However, we can check there is not such a surface by  either directly  examining the
Equations  \eqref{mo2} and \eqref{mo4} or 
by the description of these surfaces in Proposition \ref{41boundary} (see Figure \ref{41fig})
\end{proof}
\begin{cor}\label{cor:mxsblz}
Let $G_x^0$ be the connected component of the stabilizer associated to a closed orbit of a strictly semi-stable point. Then $rank(G_x^0)=1$, and up to isogeny, the largest stabilizer for a semi-stable quintic surface is either 
$\SL(2, \mathbb{C})$ or $\mathbb{G}_m$. 
\end{cor}
\begin{proof}
Theorem \ref{gitboundary} implies that there is not  a semi-stable $X$ surface with  
a $(\mathbb{C}^*)^2$ contained in its stabilizer.  
Matsushima's criterion (see \cite{matsushima1992espaces}) implies that  $G_x^0$ is a reductive group. Therefore, our corollary follows from the classification of reductive groups over the complex numbers.
\end{proof}

\subsection{Local Analysis near the GIT boundary}\label{localGIT}
Here, we discuss the local structure at selected points, in the etale topology,  of our GIT quotient 
$$
\overline{\mathcal{M}}^{\scriptscriptstyle{GIT}}_5
:=
\left( \mathbb{P}^N \right)^{ss} 
\gquot
G
$$ 
where $G \cong SL(4, \mathbb{C})$ and $N=55$. 
The main technical tool is Luna's slice Theorem \cite[App D]{luna}. Let 
$x \in (\mathbb{P}^N)^{ss}$ be a strictly semi-stable point with stabilizer $G_x$. 
There is a $G_{x}$-invariant slice $V_x$ to the orbit $G \cdot x$ which can be taken to be a smooth, affine, locally closed subvariety of $(\mathbb{P}^N)^{ss}$ such that 
$U=G \cdot V_x$ is open in $(\mathbb{P}^N)^{ss}$. Given
$(G \times_{G_x} V_x)/G_x$ where the action on the product is given by 
$h \cdot (g, v) = (g \cdot h^{-1}, hv)$, and by considering the fiber of the normal bundle 
$\mathcal{N}_x:=\left( \mathcal{N}_{G \cdot x | \mathbb{P}^n} \right) |_x $,
we have the following commutative diagram:
$$
\begin{CD}
G\times_{G_x}\mathcal{N}_x @<\text{\'etale}<<G\times_{G_x}V_x@>\text{\'etale}>>U@.\subset@. (\mathbb{P}^N)^{ss}\\
@VVV @VVV@VVV@.@VVV\\
\mathcal{N}_x\gquot G_x@<\text{\'etale}<<V_x\gquot G_x@>\text{\'etale}>>\ U\gquot G\ @.\subset@.\ (\mathbb{P}^N)^{ss} \gquot G\ @.\cong@. \overline{\mathcal{M}}^{\scriptscriptstyle{GIT}}_5
\end{CD}
$$
Recall that Kirwan constructed a partial desingularization of the GIT quotient by blowing up loci associated to positive dimensional stabilizers. 
The associated exceptional divisor $\mathbb{P}(\mathcal{N}_{x})^{ss} \gquot G_x$ often carries itself a modular meaning. It is of special interest to understand the Kirwan blow up of 
$\overline{\mathcal{M}}^{\scriptscriptstyle{GIT}}_5$ at
the point 
$
\omega
$ that parametrizes a union of a double smooth quadric surface and a transversal hyperplane. Indeed, J. Rana \cite[Thm 1.4 and 4.1]{Rana} proves that 
on the KSBA compactification
$
\overline{\mathcal{M}}^{\scriptscriptstyle{KSBA}}_5
$
 there is a Cartier divisor $\mathcal{D}$ associated with the deformations of the $\frac{1}{4}(1,1)$ singularity. At least one component of this divisor is obtained from taking the stable replacement of the following family of quintic surfaces deforming to $\omega$:
\begin{align}\label{sdq}
X_t =\big(f_2(\textbf{x})^2f_1(\textbf{x})+tf_2(\textbf{x})f_3(\textbf{x})
+t^2f_5(\textbf{x}) =0 \big)
\end{align}
where $f(\textbf{x}) :=f(x_0,x_1,x_2,x_3)$. 
By Proposition \ref{kdq} and by \cite[Thm 1.4 and 4.1]{Rana}, we can recover  Equation \eqref{sdq}  from a local analysis of the GIT quotient near $\omega$.   (For a similar situation in degree four, see Shah \cite{shah4}).

\begin{lemma} \label{lb1}
Let $\omega \in \overline{\mathcal{M}}^{\scriptscriptstyle{GIT}}_5$ be the point parametrizing a union of a double smooth quadric surface $Q$ and a transversal hyperplane $H$. 
Let $x$ be a semi-stable point with closed orbit mapping to the point $\omega \in \overline{\mathcal{M}}^{\scriptscriptstyle{GIT}}_5$. Then, the natural representation of its stabilizer $G_x \cong \operatorname{SL}(2, \mathbb{C})$ on the normal bundle $\mathcal{N}_x$ is isomorphic 
to 
\begin{align*}
\mathcal{N}_{x} &= 
\left( \operatorname{Sym}^{5}(V) \otimes \operatorname{Sym}^5(V) \right) \oplus \operatorname{Sym}^6(V) 
\end{align*}
where $V \cong H^0(\mathbb{P}^1, \mathcal{O}_{\mathbb{P}^1}(1))$ is the standard three dimensional representation of $\operatorname{SL}(2,\mathbb{C})$ induced by the conic 
$Q \cap H$.
\end{lemma}
\begin{proof} The lemma follows from calculating an appropriate 
$G_x \cong \operatorname{SL}(2, \mathbb{C})$ equivariant decomposition of the summands in the normal exact sequence
\begin{align}\label{nes}
0 \to \mathcal{T}_{G \cdot x} \to \mathcal{T}_{\mathbb{P}^{N}} \to 
\mathcal{N}_{G \cdot x| \mathbb{P}^{N}} \to 0
\end{align}
which we localize at $x$. To find the equivariant decomposition of $T_{\mathbb{P}^N}$, we use \cite[Exercise. 11.14]{fultonRp} together with 
the decomposition on Expression \eqref{dpW} 
to calculate the following decomposition of 
$H^0\big( \mathbb{P}^N, \mathcal{O}_{\mathbb{P}^N}(1) \big)$:
\begin{align*}
\operatorname{Sym}^{10}(V) \oplus \operatorname{Sym}^8(V) \oplus \left( \operatorname{Sym}^6(V) \right)^{\oplus 2}
\oplus \left( \operatorname{Sym}^4(V) \right)^{\oplus 2}
\oplus \left( \operatorname{Sym}^2(V) \right)^{\oplus 3} 
\oplus \left( \operatorname{Sym}^0(V) \right)^{\oplus 3}. 
\end{align*}
Next, we use the  Euler sequence to obtain
\begin{align}\label{eusq}
0 \to 
\mathcal{O}_{\mathbb{P}^N}|_{x}
\to 
\mathcal{O}_{\mathbb{P}^N}(1)|_{x} \otimes
H^0\big( \mathbb{P}^N , \mathcal{O}_{\mathbb{P}^N}(1) \big)
\to
\mathcal{T}_{\mathbb{P}^N}|_x
\to 
0,
\end{align}
so the decomposition of 
$H^0 ( \mathbb{P}^N , \mathcal{O}_{\mathbb{P}^3}(1) )$
induces a decomposition  at the tangent space $\mathcal{T}_{\mathbb{P}^N}|_{x} $.
To calculate the decomposition of the tangent space $\mathcal{T}_{G.x}|_x $ we use the exact sequence
\begin{align}\label{gg}
0 \to \mathcal{T}_{G_x} \to \mathcal{T}_{G} \to \mathcal{T}_{G \cdot x} \to 0.
\end{align}
The tangent space $T_{G_x}|_x$ is identified with the adjoint representation of 
$\mathfrak{sl}(2, \mathbb{C})$, which is isomorphic to $\operatorname{Sym}^2(V)$. The tangent space of $T_G|_x$ corresponds to the Lie algebra 
$\mathfrak{sl}(4, \mathbb{C})$, which has a $15$ dimensional adjoint respresentation. The embedding $C \hookrightarrow \mathbb{P}^3$ induces a decomposition as
$$
T_G|_x \cong 
\operatorname{Sym}^4(V) \oplus \left( \operatorname{Sym}^2(V) \right)^{\oplus 3} 
\oplus
\operatorname{Sym}^0(V) 
$$
from which we obtain $T_{G \cdot x} |_x$. Therefore, by comparing irreducible summands in the exact sequence \eqref{nes}, we obtain the following
decomposition for $\mathcal{N}_{G \cdot x | \mathbb{P}^n}|_x$:
$$
\operatorname{Sym}^{10}(V) \oplus \operatorname{Sym}^8(V) \oplus \left( \operatorname{Sym}^6(V) \right)^{\oplus 2}
\oplus \operatorname{Sym}^4(V) \oplus \operatorname{Sym}^2(V) \oplus \operatorname{Sym}^0(V)
$$
from which we obtain our statement by \cite[Exer. 11.11]{fultonRp}.
\end{proof}

Next, we show that a quintic surface parametrized by a point close to $\omega$ can be written in a normal form.  
Let $G_{Q}$  be the stabilizer of $Q$; the 
$G_{Q}$-equivariant decomposition 
of $H^0(\mathbb{P}^3, \mathcal{O}_{\mathbb{P}^3}(5))$
induced by the quadric surface $(F_{Q}=0)$ is
$$ 
H^0(\mathbb{P}^3, \mathcal{O}_{\mathbb{P}^3}(5)) \cong 
W_5 \oplus W_3 \oplus W_1
$$
where $W_5 \cong \operatorname{Sym}^5(V) \times \operatorname{Sym}^5(V)$ is the space of quintic surfaces intersecting the quadric surface $Q$ along a $(5,5)$ curve, 
$W_{3} \cong \operatorname{Sym}^3(V) \times \operatorname{Sym}^3(V)$ corresponds to quintic surfaces decomposing as a union of the quadric surface $Q$ and a disjoint cubic surface $(F_3=0)$, 
and $W_1 \cong \operatorname{Sym}^1(V) \times \operatorname{Sym}^1(V)$  corresponds to the quintic surfaces decomposing as a union of a double smooth quadric surface $Q^2$ and an arbitrary hyperplane. 

If the equation of the 
quintic surface and the parametrization of the invariant conic are
\begin{align*}
F_{X_0}:=x_1\left( 
x_0x_3-x_2^2 - x_1^2
\right)^2
&; &
[s_0:s_1] \to [s_0^2:0:s_0s_1:s_1^2].
\end{align*}
We can write a polynomial parametrized by
$ \left( \operatorname{Sym}^{5}(V) \otimes \operatorname{Sym}^5(V) \right) $ 
as
\begin{align*}
F
&=
\sum_{k=0}^{i=2}
(x_0x_3-x_2^2)^k
\big(
f_{5-2k}(x_0,x_2)+g_{5-2k}(x_2,x_3)+ x_1(h_{4-2k}(x_0,x_2)+p_{4-2k}(x_2,x_3))
\big).
\end{align*}
There is a $\operatorname{Sym}^6(V) \subset W_3$ associated to the intersection of $(F_3=0)$, the quadric surface $Q$, and the hyperplane $H$.  A term parametrized by $\operatorname{Sym}^6(V)$ can be written as
\begin{align*}
F_Q(x_0,x_1,x_2,x_3) G(x_0,x_1,x_2,x_3) 
= (x_0x_3-x_2^2+x_1^2) \left( f_3(x_0,x_2)+g_3(x_2,x_3) \right).
\end{align*}
The previous discussion, Luna's slice theorem, and Lemma \ref{lb1} imply the following standardization lemma (for a similar result in quartic surfaces see \cite[Lemma 4.2]{shah4}).
\begin{prop}\label{kdq}
Let $(F_{X_0}=0)$ be a union of a smooth double quadric surface and a transversal hyperplane $H$.  We may modify a given family of quintic surfaces specializing to $(F_{X_0}=0)$ such that the new family 
is defined by an equation of the form
\begin{align*}
P_t(x_0,x_1,x_2,x_3) = F_{X_0}(x_0,x_1,x_2,x_3)+F_Q(x_0,x_1,x_2,x_3)G(t, x_0,x_1,x_2,x_3)+ F(t,x_0,x_1,x_2,x_3)
\end{align*}
where 
\begin{enumerate}
\item $F(t, x_0,x_1,x_2,x_3) \in \left( \operatorname{Sym}^{5}(V) \otimes \operatorname{Sym}^5(V) \right) \otimes \mathbb{C}[[t]]$ 
and $lim_{t \to 0} F_t \neq 0$.
\item $G(t, x_0,x_1,x_2,x_3) \in \operatorname{Sym}^6(V) \otimes \mathbb{C}[[t]]$ 
and $lim_{t \to 0} G_t \neq 0$.
\end{enumerate}
Moreover, the point in $\mathbb{P}(\mathcal{N}_x)$ corresponding to 
the limits of $lim_{t \to 0} F_t$ and $lim_{t \to 0} G_t$ is semi-stable and belongs to a minimal orbit. 
\end{prop}

Next, we describe a similar analysis for the other boundary components. In the following statement, an exponent $n$ means that the corresponding entry is repeated $n$ times.

\begin{prop}\label{lb234}
The fiber of the Kirwan blow up over  $x \in \Lambda_2$ is
$$
\mathbb{P}(10,9,8,7^2,6^3,5^3,4^2,3^2,2^3,1^2) 
\times
\mathbb{P}(10,9,8,7^2,6^3,5^3,4^2,3^2,2^3,1^2).
$$
The exceptional divisor associated to the Kirwan blow up of $x \in \Lambda_3$ is $W\mathbb{P}^{18}_a \times W\mathbb{P}^{21}_b $ where
$$
W \mathbb{P}^{18}_a \cong \mathbb{P}(25,21,18,17,16,14,13,12,11,10,9,8,7,6,5,4,3,2,1)
$$
and
$$
W \mathbb{P}^{21}_b \cong \mathbb{P}(1^2,2,3,4,5,6^2,7,8,9,10^2,11^2,12,13,14,15,16,18,20).
$$
The fiber of the Kirwan blow up over $x \in \Lambda_4$ is
$$
\mathbb{P}(15,12,11,10,9,8,7^2,6^2,5,4, 3^2,2^2,1)
\times
\mathbb{P}(1^2,2^3,3^3,4^3,5^3,6^3,7^2,8^2,9,10).
$$
\end{prop}
\begin{proof} 
Let $x$ be a semi-stable point with closed orbit mapping to the GIT boundary $\Lambda_2$, $\Lambda_3$, or $\Lambda_4$. Our statement follows after finding the eigenvalues associated to the action of the stabilizer
$G_x^0 \cong \mathbb{C}^*$ on the normal bundle $\mathcal{N}_x$. Given a one-parameter subgroup 
$\lambda_k$ with $k=\{ 2,3,4 \}$, the $\lambda_k$ equivariant decomposition 
$
W \cong \bigoplus_{a_i} V_{a_i} 
$ 
induces a decomposition of the space of monomials
$
\operatorname{Sym}^5(W)
=
\bigoplus_{\alpha} V_{\alpha}^{\oplus n_{\alpha}}
$. 
We can choose the point $x$, so it parametrizes quintic surfaces given by the Equations 
\eqref{mo2}, \eqref{mo3}, and \eqref{mo4}. To calculate $\mathcal{T}_{\mathbb{P}^N}$, we use the Euler sequence \eqref{eusq}. The line bundle $\mathcal{O}_{\mathbb{P}^N}$ and $ \mathcal{O}_{\mathbb{P}^N}(1) = V_0$ has weight zero. At $x$, from the Euler sequence we obtain
$$
0 
\to 
V_0
\to 
V_0 \otimes \bigoplus_{\alpha} V_{\alpha}^{\oplus n_{\alpha}}
\to 
\mathcal{T}_{\mathbb{P}^N} |_x
\to 
0
$$
from which we obtain the decomposition of $\mathcal{T}_{\mathbb{P}^N}|_{x}$. To obtain the decomposition of
$\mathcal{T}_{G.x}|_x$, we use the exact sequence \eqref{gg}. The tangent space to $G$ is the Lie algebra $\mathfrak{sl}(4, \mathbb{C})$, and $\mathcal{T}_{G_x}$ is the adjoint representation of $\lambda_k \cong \mathbb{G}_m$. The one-parameter subgroup $\lambda_k$ acts by conjugation on $\mathfrak{sl}(4, \mathbb{C})$ with eigenvalues of the form $a_i-a_j$ for all $i,j$. Therefore, the exact sequence \eqref{gg} 
becomes
$$
0 \to V_0 \to 
\bigoplus_{i,j} V_{(a_i-a_j)}
\to \mathcal{T}_{G \cdot x}|_x \to 0.
$$
The expression of the normal bundle for each $\lambda_k$ follows from the exact sequence
\eqref{nes}.
\end{proof}

\section{Stable Isolated Singularities}\label{isosing}

In this section, we complete the characterization of stable normal quintic surfaces. We also discuss the role of invariants from singularity theory in determining the stability of these surfaces.

\subsection{Stability of Triple Point Singularities}

By Corollary \ref{cor:isolateddp}, we know that 
a quintic surface whose singularities are either  isolated double points or isolated triple point singularities with reduced tangent cone is stable. Quadruple points are unstable
by Proposition \ref{prop:q4pt}. Next, we consider triple point singularities with non-reduced tangent cone.  
\begin{rmk}
A surface singularity is of type
$\tilde{E}_8 $ if at the completion of the local ring its equation is equivalent to $z^2+x^3+y^6+tx^2y^2$ with $4t^3+27 \neq 0$.  Similarly, a singularity is of type $Z_{13}$ if its equation can be written as  $z^2+x^3y+y^6+axy^5$. 
A singularity is of type $W_{1,0}$ if its equation can be written as 
$x^4+(a_0+y)x^2y^3+y^6$ with $a_0^2 \neq 4$
 (for details see \cite{arnoldinvetiones}).
\end{rmk}

\begin{prop}\label{trpq}
Let $X$ be a normal quintic surface with a triple point singularity with non-reduced tangent cone at $p \in X$.  
We suppose that any other singularity of $X$ is either an isolated double point or a triple point with reduced tangent cone.  Let $Bl_{p}X$ be the monomial transformation of $X$ with center at $p$. 
\begin{enumerate}[label*=\arabic*.] 
\item If the tangent cone of $p \in X$ is a union of a double plane and another plane, then $X$ is non-stable if and only if $Bl_{p}X$ has either a line of singularities, 
a singularity of type $Z_{13}$, or a degeneration of it. 
\item
 If the tangent cone of $p \in X$ is a triple plane, then $X$ is non-stable if and only if $Bl_{p}X$  has either a line of singularities, a singularity 
of type $\tilde{E}_8$, a singularity of type $W_{1,0}$, or a degeneration of them. 
\end{enumerate}
\end{prop}
\begin{proof} 
We first describe representations of quintic surfaces with a triple point as double covers of $\mathbb{P}^2$. 
Let $p \in X$ be a triple point on a reduced quintic surface which contains only a finite number of lines through $p$. 
If $Bl_{p}X \to X$ is the monomial transformation of $X$ 
centered at the triple point, then  we have a natural morphism  $Bl_{p}X \to \mathbb{P}^2$. Consider its Stein factorization $Bl_{p}X \to X^* \to \mathbb{P}^2$, such that $X^*$ is a double cover of $\mathbb{P}^2$ branched over an octic plane curve $B(X)$. If the equation associated to the quintic surface is 
$$
F_X(x_0,x_1,x_2,x_3):= x_3^2f_3(x_0,x_1,x_2) +x_3f_4(x_0,x_1,x_2)+f_5(x_0,x_1,x_2),
$$
then
\begin{align}\label{Bx}
F_{B(X)}= f_3(x_0,x_1,x_2)f_5(x_0,x_1,x_2) -f_4(x_0,x_1,x_2)^2.
\end{align}
The map $Bl_{p}X  \to X^*$ contracts the proper transform of the lines $ L \subset X$ through the triple point, and it is an isomorphism everywhere else. 
Thus, if there is no a line in $X$ passing through $p$, it holds $Bl_{p}X \cong X^*$.
If the singularity of $X$ is supported at $p$, then the singularities on $Bl_{p}X$ are supported in the exceptional divisor, E, of the monomial transformation. The reduced image of 
$E$  in $\mathbb{P}^2$ is the curve defined by  
$f_3(x_0,x_1,x_2)=0$. By using partial derivatives, we can see that the singularities of $Bl_{p}X$  are supported at
$$
\operatorname{Sing} \big( f_3(x_0,x_1,x_2) =0 \big) \cap \big( f_4(x_0,x_1,x_2) =0 \big).
$$
Next, we prove our results.

First, we suppose that $X$ is non-stable.  By our hypotheses and the results of  Section \ref{sec:GIT}, 
there is a change of coordinates such that  $\Xi_{F_X} \subset M^{\oplus}(\lambda_i)$ for 
$i \in \{ 1,3,4 \}$ (see Table \ref{cr1ps}).  From the equations in the proof of Proposition \ref{prop:GIT}, we obtain the following:
\begin{enumerate}
\item If $\Xi_{F_X} \subseteq M^{\oplus}(\lambda_1)$, then 
the tangent cone of $T_pX$ is a union of a double plane and another plane, and
$Bl_pX$ has a line of singularities. 
\item If $\Xi_{F_X} \subseteq M^{\oplus}(\lambda_3)$, then we have two options:
First, the tangent cone of $T_pX$ is a union of a double plane and a different plane; and
$Bl_pX$ has either a singularity of type $Z_{13}$ or a degeneration of it.
Second,  the tangent cone of $T_pX$ is a triple plane and $Bl_pX$ has either a singularity of type $W_{1,0}$ or a degeneration of it.
\item  If $\Xi_{F_X} \subseteq M^{\oplus}(\lambda_4)$, then 
the tangent cone of $T_pX$ is a triple plane and  $Bl_pX$ has either a $\tilde E_8$ singularity
or a degeneration of it.
\end{enumerate}
If $X$ is a non-stable normal quintic surface with a triple point singularity,  then we can find a general deformation of $X$ that preserves the type of $M^{\oplus}(\lambda_i)$.   
It follows that $X$ has the described singularities. 

Next, we show that if $Bl_pX$ has the mentioned singularities, then $X$ is non-stable. The case whenever  $Bl_p X$ is non-normal follows from Corollary \ref{convL1}.  Now we suppose that  $Bl_pX$ has a singularity of type $Z_{13}$  and the tangent cone of $X$ at $p$ is the union of a plane and another one.  We may assume the triple point is supported at $p_3$, and its tangent cone is given by $x_0^2l(x_0,x_1)$. So every point in $B(X)$ inducing a singular point in $Bl_{p}X$ is supported at $(x_0=0) \cap B(X)$. There is a point 
$q$ in $(x_0=0) \cap B(X)$ of multiplicity at least four
because $Bl_{p}X$ has a non-ADE singularity. We can take $q$ to be supported at $[0:0:1]$. 
The normal form of $Z_{13}$ implies that 
the tangent cone of $B(X)$ at $[0:0:1]$ can be taken to be 
$x_0^3g_1(x_0,x_1)$. Then the equation of $B(X)$ can be written as
\begin{align*}
F_{B(X)} & = x_0^2f_1(x_0,x_1)\left(
x_2^4x_0+x_2^3f_2(x_0,x_1)+x_2^2f_3(x_0,x_1)+x_2f_4(x_0,x_1)+f_5(x_0,x_1)
\right) 
\\
& \qquad{}
+
\left( x_2^2x_0^2+x_2f_3(x_0,x_1)+f_4(x_0,x_1) \right)^2.
\end{align*}
By comparing $F_{B(X)}$ to $F_{\lambda_3}$, we obtain that $\Xi_{F_X} =M^{\oplus}(\lambda_3)$.

Now we suppose that the tangent cone of $p \in X$ is a triple plane and $Bl_pX$
as an $\tilde E_8$ singularity.  We may select a coordinate system such that the triple point of $X$ is supported   at $[0:0:0:1]$,  and the tangent cone is supported at $(x_0=0)$.
By our hypothesis, $Bl_pX$ has a $\tilde E_8$ singularity. Then 
 $B(X)$ has a semi-quasi-homogeneous singularity of degree $6$ with respect to the weights $w(x)=3$ and $w(y) =2$ at $[0:0:1]$. These weights determine the $\tilde{E}_8$ singularity
in the double cover. The most general equation for such an octic plane curve can be written as 
\begin{align*}
F_{B_X}= 
x_0^3f_5(x_0,x_1,x_2)
-(x_2^2x_0f_1(x_0,x_1)+x_2f_3(x_0,x_1)+f_4(x_0,x_1))^2.
\end{align*}
By comparing $F_{B(X)}$ to $F_{\lambda_4}$, we obtain that $\Xi_{F_X} =M^{\oplus}(\lambda_4)$.

Now we suppose that the tangent cone of $p \in X$ is a triple plane, and $Bl_pX$
has an $W_{1,0}$ singularity.  By the same argument as  above, we find that 
\begin{align*}
F_{B(X)} & = x_0^3\left(
x_2^4x_0+x_2^3f_2(x_0,x_1)+x_2^2f_3(x_0,x_1)+x_2f_4(x_0,x_1)+f_5(x_0,x_1)
\right) 
\\
& \qquad{}
+
\left( x_2^2x_0^2+x_2f_3(x_0,x_1)+f_4(x_0,x_1) \right)^2,
\end{align*}
and we obtain that $\Xi_{F_X} \subsetneq M^{\oplus}(\lambda_3)$.

From the above discussion, it follows that if $Bl_{p}X$ has a singularity as described in the statement, then  there is a coordinate system such that  $ \Xi_{F_X}$ is contained in $M^{\oplus}(\lambda_i)$ for $i \in \{1,3,4 \}$, and $X$ is non-stable.
\end{proof}

\subsection{Invariants of Singularities and GIT Stability}

Next, we relate the stability of normal quintic surfaces  to the study of invariants associated to their singularities. We start with  Milnor number and modality. They are invariants used in the classification of singularities due to Arnold \cite{arnoldinvetiones}, Suzuki, Yoshinaga \cite{watanabeyoshi,m5}, and Estrada et al. \cite{m6}. 
\begin{prop} \label{boundinv}
If $X$ is a normal quintic surface where each of its singularities has either Milnor number smaller than 22 or modality smaller than 5,  then $X$ is stable. 
\end{prop}
\begin{proof}
We prove it by contradiction. If the surface $X$ is not stable, then the GIT analysis implies there is a coordinate system such that $\Xi_{F_X}$ is contained in one of the 
$M^{\oplus}(\lambda_i)$ for $\lambda_i$ as in Proposition \ref{prop:maximalsemistableops}. In particular, a destabilizing isolated singularity of $X$ is supported at $p_3$, and it deforms to the singularity of $(F_{\lambda_i}=0)$ for $\lambda_1$, $\lambda_3$, $\lambda_4$, or $\lambda_7$ (see Table \ref{cr1ps}).
We consider a general deformation of $X$ that preserves the type $M^{\oplus}(\lambda_i)$ with respect to the given choice of coordinates. By Theorem \ref{gitboundary}, the singularities at $(F_{\lambda_i}=0)$ are either $V_{24}^{*}$ (notation as \cite{m6}),  $V_{24}^{*1}$ (notation as \cite{m5}),  $V_{22}^{\prime}$ (notation as \cite{m5}), or an ordinary quadruple point. 

Now we use the fact that the Milnor number of the $V_{24}^{*}$ singularity is 24 and its modality is 7. The Milnor number of the $V_{24}^{*1}$ singularity is 24 and its modality is 5. 
The Milnor number of the $V_{22}^{\prime}$  singularity is 22 and its modality is 5.
The Milnor number of a quadruple point is at least $27$ and its modality is at least $6$ (see \cite{m5}, \cite{m6}). Therefore, the statement follows by the upper semi-continuity of both the Milnor number and  modality.
\end{proof}
\begin{ex}  
The previous bound  in the Milnor number  is  not a necessary condition for stability. 
Indeed, at $(t \neq 1)$ the zero set of the equation
$$
F_t(x_0,x_1,x_2,x_3)= x_0^3x_3^2+2x_2x_1^3x_3 +x_2^5
+t \big( x_3^3x_2^2-3x_0^2x_1^2x_3+
3x_0x_1^4+x_2^3x_0^2 \big)
$$
has a weakly elliptic singularity at $p_3$, which is formally equivalent to the singularity induced by the equation $x^2+y^3+z^{13}$ 
(see \cite[pg 452]{yang}). This singularity has Milnor number equal to $24$
and is stable by Corollary \ref{cor:isolateddp}. The zero set 
$(F_0(x_0,x_1,x_2,x_3)=0)$ is a non-normal surface parametrized by  $\Lambda_4$.
\end{ex}
For a non-log-canonical singularity $p\in X$, the log-canonical threshold $c_p(X)$ is an invariant valued between $0$ and $1$ such that the smaller its value, the worse  the singularity (see \cite[pg. 45]{singpairs} for definitions and details). 
The relationship between the log-canonical threshold and the GIT stability given in 
Lemma \ref{prop:lctstable} below was first noticed by Hacking \cite[Prop 10.4]{hacking} and Kim  and Lee \cite[Rmk. 2.4]{kim}.
\begin{lemma}
\label{prop:lctstable}
A normal quintic surface having at worst a singularity with log-canonical threshold 
(equal to) greater than $4/5$ is (semi-) stable. 
\end{lemma}
\begin{rmk}
The converse of \cite[Rmk. 2.4]{kim} 
does not hold in general. For example, there are semi-stable quartic plane curves with an 
$A_5$ singularity.
\end{rmk}
Next,  we describe   a natural family of singularities, called minimal elliptic, with a
 log-canonical threshold greater than $4/5$. 
\begin{defn}\label{pg}
Let $X$ be a normal surface singular at $p$, the geometric genus of the singularity is
$\dim( R^1\pi_{*} \mathcal{O}_Y)$ where
$\pi: Y \to X$ is a resolution of $X$ at $p$. 
\end{defn}
This invariant induces a well-known classification of singularities: \emph{Rational singularities} are those for which the geometric genus is zero. For surfaces, the rational Gorenstein surface singularities are the ADE ones. After rational surface singularities, we find the family of minimal elliptic singularities classified by Laufer
\cite{laufer}. Next, we provide not the original definition of minimal elliptic singularities, but rather a convenient one. Recall that we work with isolated hypersurface singularities which are always Gorenstein.
\begin{defn}
\label{defn:minimalelliptic}
(\cite[Th. 3.10] {laufer})
A surface singularity is minimal elliptic if and only if it is Gorenstein and 
$ 
\dim R^1\pi_* \left( \mathcal{O}_Y \right) 
=1
$.
\end{defn} 
An important application of the log-canonical threshold criterion is the GIT stability of the minimal elliptic singularities.
\begin{prop}\label{prop:stableminimalelliptic}
Let $X \subset \p{3}$ be a normal quintic surface with either ADE or  minimal elliptic singularities. Then $X$ is stable. 
\end{prop}
\begin{proof} 
An analysis of the log-canonical threshold for minimal elliptic singularities in a hypersurface in $\mathbb{P}^3$ is done by Prokhorov in \cite[Table 1-3]{prokhorov}. In particular, their log-canonical value is greater than or equal to $(\frac{4}{5}+\frac{1}{180})$. Therefore, they are GIT stable by Lemma
\ref{prop:lctstable}. 
\end{proof}
The genus of a singularity $p\in X$ can be interpreted by its effect on the geometric genus, $p_g(X)$, of the variety $X$. We include a proof for completeness.
\begin{prop}\label{effectgenus}
Given the minimal resolution $\pi:Y \to X$ of a normal hypersurface of degree $d$, with a unique non-ADE singularity of genus $R^1 \left( \pi_{*} \mathcal{O}_Y \right)$, 
we  have
$$
\frac{(d-1)(d-2)(d-3)}{6} - p_g(Y)+ q(Y)=R^1 \left( \pi_{*} \mathcal{O}_Y \right).
$$ 
Furthermore, if $X$ is a quintic surface and $Y$ is of general type, then $q(Y)=0$, and we have
$$
4 -p_g(Y)=R^1 \left( \pi_{*} \mathcal{O}_Y \right).
$$
\end{prop}
\begin{proof}
On a normal hypersurface $X$ of degree $d \geq 4$, we have $H^1(X, \mathcal{O}_X) = q(X)=0$ and 
$$
H^2(X, \mathcal{O}_X)=p_g(X)=(d-1)(d-2)(d-3)/6.
$$ 
From those values and the exact sequence
(see \cite[pg. 433 ]{yang})
$$
0 \to H^1 \left( X, \mathcal{O}_{X} \right) \to H^1 \left( Y, \mathcal{O}_{Y} \right) \to 
R^1 \pi_{*} \mathcal{O}_Y 
\to H^2 \left(X , \mathcal{O}_{X} \right) \to H^2 \left( Y, \mathcal{O}_{Y} \right) \to 0,
$$
we obtain
$
p_g(X) - \dim R^1 \left( \pi_{*} \mathcal{O}_Y \right) = p_g(Y)-q(Y)=p_a(Y)
$.
If $X$ is a quintic surface and  $Y$ is of general type, then the irregularity $q(Y)$ vanishes by \cite{umezu}.
\end{proof}
The following result  illustrates the complexity of the  singularities
in a semi-stable surface. 
\begin{prop}\label{l:generalpg}
There is at least one semi-stable hypersurface $X \subset \mathbb{P}^3$ of degree $d \geq 4$ with an isolated quasi-homogeneous singularity of genus $g(d)$ where
\begin{align*}
g(d)=
\begin{cases}
\frac{d(d-2)(4d-10)}{48} & \text{ if $d$ is even} 
 \\
&
\\
\frac{(d-1)(d-3)(4d-2)}{48} & \text{ if $d$ is odd.} 
\end{cases}
\end{align*}
Note that $g(4)=1$, $g(5)=3$, and $g(6)=7$. 
\end{prop}
\begin{proof}
From the combinatorics of the GIT setting, it is clear that for any degree $d \geq 4$ 
the one-parameter subgroup $\lambda_1 =(1,0,0,-1)$ is 
always  critical  and $M^{\oplus}(\lambda_1)$ is a maximal semi-stable set.   The  generic associated semi-stable surface  is the zero set of  the polynomial 
$F_{\lambda_1}(x_0,x_1,x_2,x_3)$  stabilized by $\lambda_1$.
If $d=2m+1$, then
$$
F_{\lambda_1}(x_0,x_1,x_2,x_3)=
x_3^mx_0^mf_1(x_1,x_2) + x_3^{m-1}x_0^{m-1}f_3(x_1,x_2)
+ \ldots +f_{2m+1}(x_1,x_2),
$$
and a similar equation is associated to the case $d=2m$.  After localizing, we have 
a quasi-homogeneous polynomial with weights 
$w=(2,1,1)$,  and weighted multiplicity $d$, and  whose weighted leading term induces an isolated singularity.
The geometric genus of a quasi-homogeneous isolated singularity hypersurface is determined by its weights.  We apply Lemma \cite[pg 48]{watanabeyoshi} using the expressions
$n_i/d =(a_i-a_3)/w(f)$,  
where $(a_0,a_1,a_2,a_3)=(1,0,0,-1)$ and 
$w(f_p) - d=0$. Therefore, the geometric genus of the singularity at $F_{\lambda_1}$ is given by the number of non-negative integer solutions of the following equations:
\begin{align}\label{dioeq}
d & = i_0+i_1+i_2+i_3 \\
|a_3|(d-4)
& \geq 
(a_0-a_3)i_0+(a_1-a_3)i_1+ (a_2-a_3)i_2.
\notag
\end{align}
This is calculated by a standard method which was shown to us by E. Rosu. From it, we find that the geometric genus is equal to 
$$
\sum_{k=0}^{\left[ \frac{d-4}{2} \right]} \binom{d-2-2k}{2}.
$$
This formula becomes the expression of the statement after some algebraic manipulations.
\end{proof}

\section{Stability and Non-Isolated Singularities}
\label{sec:nonisolatedsing}

In this section, we give a partial description of the stability for reducible quintic surfaces,  surfaces with a certain curve of singularities of multiplicity three, and quintic surfaces with  a line of singularities of multiplicity two.   These cases will complement the ones discussed in previous sections (see  Proposition \ref{prop:tripleline},  Corollary \ref{highergenus}, and Corollary \ref{2qh}).  

\subsection{Reducible Quintic Surfaces}
A generic quintic surface that decomposes as a union of a quartic surface and a hyperplane is  GIT stable.  The locus, called 
$\overline{M}(4,1)$, in the GIT quotient that parametrizes those surfaces is twenty two dimensional: Nineteen dimensions are associated to the moduli of $K3$ surfaces, and three dimensions are associated to the choice of a hyperplane in $\mathbb{P}^3$.
\begin{prop}\label{41}
Let $X$ be a quintic surface that decomposes as a union of a hyperplane $H$ and a quartic normal surface $Y$ with isolated singularities such that
\begin{enumerate}[label*=\arabic*.]
\item The isolated singularities in our quartic surface does not destabilize the quintic surface.
\item The singular locus of the quartic surface $Y$ is disjoint from the hyperplane, and the quartic plane curve $Y \cap H$ has at worst a triple point whose tangent cone has a double line.
\end{enumerate}
Then $X$ is stable. 
\end{prop}
\begin{proof}
If the quintic surface  is non-stable, then there is a coordinate system and a normalized 1-PS $\lambda$ such that 
$\mu(X, \lambda) \geq 0$. By condition (1) in the statement,  $p_{\lambda}$ must be supported in the intersection of the hyperplane with the quartic surface; and $p_{\lambda}$ must have multiplicity two because $Sing(Y)$ is disjoint from $H$. 
By our results in Section \ref{sec:GIT},  $\Xi_{F_X}$ is contained in $M^{\oplus}(\lambda_k)$ for $k=5,9,10$ (see Table \ref{cr1ps}).
However, this is not possible according to our hypothesis about $Y \cap H$, 
the fourth case of Proposition \ref{prop:GIT}, and the third and the fourth cases of Proposition \ref{prop:unstableGIT}. 
\end{proof}

Next, we describe the intersection between $\overline{M}(4,1)$ and our GIT boundary.
We can say, somewhat informally, that these are the worst unions of a quartic surface and a hyperplane parametrized by our GIT quotient.
\begin{prop}\label{41boundary}
Let $X$ be a quintic surface parametrized by a point in the intersection between the locus $\overline{M}(4,1) $ and $\Lambda_i$.  
Then one of the following conditions holds:
\begin{enumerate}[label*=\arabic*.]
\item The surface $X$ is parametrized by $\Lambda_1$ and satisfies the following:
\begin{itemize}
\item[*] 
The quartic surface has two $\tilde{E}_7$ singularities.
\item[*] The intersection of the hyperplane and the quartic surface is a union of two conics of the form
$$
\left( (xy-a_1z^2)(xy-a_2z^2) =0 \right).
$$
\item[*] 
The hyperplane does not intersect the singularities along their tangent cones. 
\end{itemize}
\item The surface $X$ is parametrized by $\Lambda_2$ and satisfies the following:
\begin{enumerate}[label*=\arabic*.]
\item[*] 
The singular locus of the quartic surface decomposes as a union of two coplanar double lines $L_1$ and $L_2$ intersecting at a non-isolated triple point with an associated equation of the form 
$
x^2y+x^3z+y^2z^2
$.
\item[*] 
The intersection of the hyperplane and the quartic surface decomposes as a union of a cuspidal plane curve and a line. That line is contained in the quartic surface. The singularity of the cuspidal curve is away from the triple point.
\end{enumerate}
\item
The surface $X$ is parametrized by $\Lambda_4$ and satisfies the following:
\begin{enumerate}[label*=\arabic*.]
\item[*] The singular locus of the quartic surface has a double line $L$ and a distinguished triple point given by the equation $x^3-xyz^2+zy^3$ which is away from the hyperplane.
\item[*] The intersection of the hyperplane and the quartic surface is a union of two lines and a conic tangent to one of them. 
\end{enumerate}
\end{enumerate}
We represent those geometric characteristics in Figure \ref{41fig}.
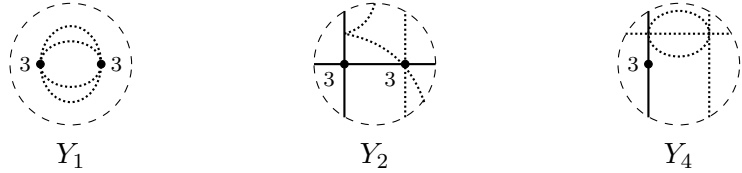
\begin{figure}[h!]
\centering
\begin{tikzpicture}
\draw[dashed] (0,0) circle [radius=0.8cm];
\filldraw (-0.4,0) circle [radius=0.05cm];
\draw (-0.6,0) node {$\scriptstyle{3}$};
\filldraw (0.4,0) circle [radius=0.05cm];
\draw (0.6,0) node {$\scriptstyle{3}$};
\draw (0,-1.2) node {$Y_1$};
\draw[thick, densely dotted] (0,0) circle [x radius=0.4cm, y radius=0.3cm];
\draw[thick, densely dotted] (0,0) circle [x radius=0.4cm, y radius=0.5cm];
\draw[dashed] (4,0) circle [radius=0.8cm];
\draw[thick] (3.2,0) - - (4.8,0);
\draw[thick] (3.6,-0.7) - - (3.6,0.7);
\draw[thick, densely dotted] (4.4,-0.7) - - (4.4,0.7);
\filldraw (3.6,0) circle [radius=0.05cm];
\draw (3.4,-0.2) node {$\scriptstyle{3}$};
\filldraw (4.4,0) circle [radius=0.05cm];
\draw (4.2,-0.2) node {$\scriptstyle{3}$};
\draw[thick, densely dotted] (4.65,-0.5) to[bend right] (3.6,0.4);
\draw[thick, densely dotted] (4,0.8) to[bend left] (3.6,0.4);
\draw (4,-1.2) node {$Y_2$};
\draw[dashed] (8,0) circle [radius=0.8cm];
\draw[thick, densely dotted] (7.3,0.4) - - (8.7,0.4);
\draw[thick, densely dotted] (8.4,-0.7) - - (8.4,0.7);
\draw[thick, densely dotted] (8,0.4) circle[x radius=0.4cm, y radius=0.3cm];
\draw[thick] (7.6,-0.7) - - (7.6,0.7);
\filldraw (7.6,0) circle [radius=0.05cm];
\draw (7.4,0) node {$\scriptstyle{3}$};
\draw (8,-1.2) node {$Y_4$};
\end{tikzpicture}
\caption{
$Y_i$ are our  quartic surfaces, the dotted lines are the intersection $Y_i \cap H$,   bold lines are the singular locus of $Y_i$, and the numbers are the multiplicity 
of the singularities at those points.}\label{41fig}
\end{figure}
\end{prop}
\begin{proof}
Let $X$ be such a quintic surface. Then there is a one-parameter subgroup $\lambda$ such that $X$ is invariant under the action of it.  By construction $X= Y \cup H$ and it is easily seen that the hyperplane is also invariant under the action of $\lambda$. In particular, this implies that in our coordinate system the equation associated to $H$ must be $(x_i=0)$. 
From our results in Section \ref{sec:minimalorbit}, and up to a change of coordinates, we have the equations of these surfaces.  So, the statement reduces to describing their geometric characteristics which follow from their equations:
\begin{align*}
\tilde{F}_{\lambda_2} &= 
x_0(x_3^2x_1^2+x_3x_0x_2^2+x_1x_2^3)
& & 
\tilde{F}_{\lambda_4} =
x_3(x_0^3x_3+x_2x_1^3+x_0x_1x_2^2)
\\
\tilde{F}_{\lambda_1} &=
x_1\left( 
x_3^2x_0^2+
x_0x_3f_2(x_1,x_2)+f_4(x_1,x_2) \right)
\end{align*}
\end{proof}
Next, we show that a quintic surface with  a non-linear curve of singularities of multiplicity  three decomposes as a union of a quartic surface and a hyperplane, and it is generically stable. 
\begin{prop} 
\label{mltp3} 
Let $X$ be a quintic surface with a curve of singularities $C$ such that $C$ does not contain a line and $mult_p(X) = 3$ for every $p \in C$. Then $X$ decomposes as a union of a hyperplane and a quartic surface, and there is a coordinate system such that its 
associated equation can be written as 
$$
x_i\left( f_2(x_j,x_k,x_3)^2 
+
x_i^2g_2(x_0,x_1,x_2,x_3) + x_if_2(x_j,x_k,x_3) f_1(x_0,x_1,x_2,x_3)
\right).
$$
Moreover, this surface is generically stable (compare this with Proposition \ref{prop:tripleline}).
\end{prop}
\begin{proof}
Let $C$ be such a curve. Consider two generic distinct points $p$ and $q$ on it, and let $L_{p,q}$ be the line that join them. Since $p$ and $q$ are triple points,  $L_{p,q}$ intersects $X$ with multiplicity greater than or equal to six. However, since $X$ is a quintic surface, this implies that the surface contains the line $L_{p,q}$ for every $p$ and $q$ on $C$. Then, $X$ contains the secant variety $Sec(C)$ of $C$. For a curve $C$ in $\mathbb{P}^3$, the secant variety of $C$ is either the whole $\mathbb{P}^3$ or a hyperplane, with the latter option only happening if $C$ is a plane curve itself
(see \cite[pg 144]{harrisAG}). Then $C$ is a plane curve, and $X$ decomposes as a hyperplane $H$ and a quartic surface $Y$.
Moreover, from the hypotheses and by degree considerations, $C$ is a smooth conic.
Let our coordinate system be such that the critical one-parameter subgroups are the ones in Proposition \ref{prop:maximalsemistableops} and the hyperplane is given by some  $(x_i=0)$.  Then, the equation associated to the quintic surface can be written as
$$
x_i \left( f_4(x_j,x_k,x_l) + x_ig_3(x_0,x_1,x_2,x_3) \right).
$$
By our hypotheses, $m_p(X) = 3$ for every point $p \in C \subset Y \cap H$ and $C$ does not contain a line. Then, it holds that $f_4(x_j,x_k,x_l) = (f_2(x_j,x_k,x_l))^2$ and either $x_i$ or $f_2(x_j,x_k,x_l)$ divides $g_3(x_0,x_1,x_2,x_3)$. 
In our coordinate system, the most general equation satisfying these properties is the one of the statement.

Given a normalized one-parameter subgroup $\lambda=(a_0,a_1,a_2,a_3)$, we have
$$
\mu(\lambda, X)
\leq \min \{
a_i+2 \mu(\lambda, f_2), 3a_i+ \mu(\lambda, g_2),2a_i+\mu(\lambda, f_2)
+\mu(\lambda, l)
\}.
$$
In our coordinate system, the curve cannot be supported at $(x_3=0)$ because a triple point is supported at both $p_{3}$ and $H$. By the construction and smoothness of $C$, we have
$$ 
f_2(x_j,x_l,x_3)=
x_3l(x_j,x_l)+p_2(x_j,x_l)
$$
with the set of monomials 
$\Xi_{f_2}$ containing at least $\{ x_3x_j, x_l^2 \}$ with $j \neq l$ and $j,l \neq i$. 
Additionally,  it holds generically that
$\mu(\lambda, g_2) \leq 2a_1$. Therefore, 
\begin{align*}
\resizebox{1.0\hsize}{!}{
$
\mu(\lambda, X)
\leq 
\min \{
a_i+2(a_3+a_j), 
a_i+4a_l, 3a_i+ 2a_1,
2a_i+(a_3+a_j)+a_0,
2a_i+2a_l+a_1
\}
$ } 
\end{align*} 
A direct calculation  shows that   $\mu(\lambda_k, X) \leq 0$.  Thus, $X$ is semi-stable. 
\end{proof}

Next we consider a quintic surface that decomposes as a union of a cubic and a quadric surface. On the moduli space, the locus that parametrizes these surfaces is thirteen dimensional: 
Nine dimensions arise from the genus four curve defined by the intersection of the cubic and the quadric surface. The other four dimensions arise from the fact that we can add a multiple of the quadratic equation to the cubic surface equation without changing the genus 4 curve. 
\begin{prop}\label{con3q}
Let $X$ be a union of a smooth quadric surface $Q$ and a cubic surface $Y$ with a triple point at $p \notin Q$. This triple point destabilizes the quintic surface if and only if 
the tangent cone of  $Y$ at $p$ is either a union of a conic and tangent line or a degeneration of it.
\end{prop}
\begin{proof} 
Suppose $X$ is not stable, and our coordinate system is such that the critical 1-PS are the ones in Proposition \ref{prop:maximalsemistableops} and the triple point is supported at $p_{3}$. The cubic surface is a cone over a plane cubic curve $C$, 
and the  equation  of the quintic surface is
$
f_3(x_0,x_1,x_2)g_2(x_0,x_1,x_2,x_3).
$ 
By our hypothesis, the quadratic surface is away from the triple point. Therefore, the monomial $x_3^2$ is always present in $\Xi_{g_2}$, which implies
$\mu(\lambda,X) = 2a_3 +\mu(\lambda, f_3)$. The following analysis is divided by the singularities of the cubic curve.
\begin{enumerate}
\item If $C$ has a triple point, then  $X$ is unstable 
because it has either a triple line or a double plane (see  Proposition \ref{prop:tripleline} or 
Proposition \ref{2qh}).
\item If $C$ is a union of a conic with a tangent line, 
then the equation of the quintic surface can be written as
$ 
F_X=x_0(x_2x_0-x_1^2)f_2(x_0,x_1,x_2,x_3),
$ 
which is destabilized by $\lambda_9$.  
\item If $C$ is a union of three non-concurrent lines, then an  equation of the quintic surface is
$$
f_1(x_0,x_1,x_2)g_1(x_0,x_1,x_2)h_1(x_0,x_1,x_2)f_2(x_0,x_1,x_2,x_3),
$$
and the monomial $ x_0x_1x_2x_3^2$ must have coefficient different  to zero because the lines are not concurrent.  The presence of this monomial implies that 
$
\mu(\lambda_k, X) < 0
$
for all  $\lambda_k$.  
\item  If $C$ is a union of a conic and a transversal line, then it deforms to three non-concurrent lines and the stability of $X$ follows by the previous case.
\item 
By considering the partial order among monomials (see 
Section \ref{sec:GIT}), 
if $C$ has a cuspidal singularity, then in our coordinate system any
surface $X$, as in the statement, satisfies
$
\mu(\lambda, X) \leq \mu( \lambda, X_0)
$
where 
$
F_{X_0}= \left(x_i^2x_j+x_0^3+p_3(x_0,x_i) \right)g_2(x_0,x_1,x_2,x_3)
$.
The statement follows from the inequality
$\mu(\lambda_k, X_0) \leq \min\{ 2a_i+a_j+2a_3, 3a_0+2a_3  | i,j \neq 0 \} < 0$.
\item 
Finally, if $C$ is a node, then $C$ deforms to a curve with cuspidal singularity and 
the statement follows by the previous case. 
\end{enumerate}
\end{proof}

\subsection{Quintic Surfaces with a Curve of Singularities of Multiplicity Two}
 
Given a quintic surface with a curve of multiplicities two, we use a sequence of blow ups for constructing a triple cover of $\mathbb{P}^2$ branched over a curve of degree 12. The purpose is to  describe their general form and illustrate the diversity of singular surfaces parametrized by our quotient.

\begin{lemma}\label{dpcov}
Let $X$ be a quintic surface defined by the equation
\begin{align}\label{qdp}
F_{X}:=x_3^3x_0^2 +x_3^2x_0g_2(x_0,x_1,x_2)+x_3f_4(x_0,x_1,x_2)+f_5(x_0,x_1,x_2),
\end{align}
and consider the surface in 
$\mathbb{P}(2,1,1,1)$ defined by the equation
\begin{align*}
G_{F_X} = \psi^3 + \left( f_4-\frac{g_2^2}{3}\right) \psi +
\left(
x_0f_5+\frac{2}{27}g_2^3-\frac{g_2f_4}{3}
\right)
:= \psi^3 + h_4(x_0,x_1,x_2) \psi + h_6(x_0,x_1,x_3).
\end{align*}
The following statements are equivalent:
\begin{itemize}
\item[(i)] 
The set $\Xi_{F_X}$ is contained in either 
$M^{\oplus}(\lambda_5)$ or $M^{\oplus}(\lambda_9)$.
\item[(ii)] 
The polynomials $h_4(x_0,x_1,x_2)$ and $h_6(x_0,x_1,x_2)$  are obtained from a linear combination of the monomials in the sets
\begin{align}\label{wG}
\Xi_{h_4} 
&=
\left\{ 
x_0^{j_0}x_1^{j_1}x_2^{j_2} \; | \; w_0j_0+w_1j_1+w_2j_2 \geq c_1(k) \; ; \; j_0+j_1+j_2=4
\right\} 
\\
\Xi_{h_6}
&=
\left\{ 
x_0^{j_0}x_1^{j_1}x_2^{j_2} \; | \; w_0j_0+w_1j_2+w_2j_2 \geq c_2(k)
\; ; \;
j_0+j_1+j_2=6
\right\} \notag
\end{align} 
where $w_{5}=(5,2,1)$, $c_1(5)=10$,  $c_2(5)=15$ for the case  $\lambda_5$, 
and  $w_{9}=(11,5,0)$, $c_1(9)=20$, $c_2(9)=30$ for the case $\lambda_9$.
\end{itemize}
\end{lemma}
\begin{proof}
We recall the representation of quintic surfaces with a double point as a finite cover of the plane (see \cite[pg 471]{yang}). 
Let $\tilde{X} \to X$ be the monomial transformation of $X$ with center at $p=[0:0:0:1]$. There is a morphism $\tilde{X} \to \mathbb{P}^2$ induced by the projection from the point $p \in X$ that is generically finite of degree three. The surface $\tilde{X}$ is given by the equation 
$$
t^3x_0^2 +t^2sx_0g_2(x_0,x_1,x_2)+ts^2f_4(x_0,x_1,x_2)+s^3f_5(x_0,x_1,x_2)
$$
with $[ t:s ] \in \mathbb{P}^1$. From the equation, we see $\tilde{X}$ is singular along the line $(s=x_0=0)$. Blowing up $\tilde{X}$ along this line in one of the charts, the total transform $X'$ is given by
\begin{align*}
x_0^2 \left( t^3+t^2sg_2(x_0,x_1,x_2)+ts^2f_4(x_0,x_1,x_2)+x_0s^3f_5(x_0,x_1,x_2) \right).
\end{align*}
In this chart,  we  take $\psi =t/s$ and substitute 
$\psi$ by $\psi - g_2(x_0,x_1,x_2)$ to obtain the Equation $G_{F_X}$.

\textbf{Claim: $(i)$ implies $(ii)$.}  We suppose that $\Xi_{F_X} \subset M^{\oplus}(\lambda_k)$ for $k \in \{ 5,9\}$.  Given a polynomial $h_d$ of degree $d$, we denote its set of non-zero monomials as $\Xi_{h_d}$.  For $\lambda_k$ with $k \in \{ 5,9 \}$  and the monomial  $x_0^{i_0}x_1^{i_1}x_2^{i_2}x_3^{i_3}$ (which we denote as $[i_0,i_1,i_2,i_3]$) with $i_0+i_1+i_2+i_3=5$, it holds that
\begin{align}\label{lw}
\lambda_k . [i_0,i_1,i_2,i_3] = w_{k}(i_0,i_1,i_2)+5a_3.
\end{align}
where $w_{k}(i_0,i_1,i_2)$ is  equal to the weighted degree of the monomial $[i_0,i_1,i_2]$.  
By Equation \eqref{lw} and simple arithmetic, we find that 
the following statements are equivalent
\begin{itemize}
\item[(*)] The set $M^{\oplus}(\lambda_k)$ contains $\Xi_{x_3f_4} $, $\Xi_{x_3^2x_0g_2}$ and 
$\Xi_{f_5}$
\item[(*)]
 The weighted of degree of  $h_4$ and $h_6$ satisfy
$w_{k}(h_4) \geq c_1(k)$ and
$w_{k}(h_6) \geq c_2(k)$
\end{itemize}
from which our claim follows.

\textbf{Claim: $(ii)$ implies $(i)$.} We suppose the monomials
$[i_0,i_1,i_2]$ in $\Xi_{h_4}$ and $\Xi_{h_6}$ satisfy condition $(ii)$, and there is a quintic surface  such that the equation of $G_{F_X}$ is induced by $F_X$.
Let $m=[i_0,i_1,i_2]$ be a monomial in $\Xi_{h_4}$, by construction either
$m \in \Xi_{f_4}$ or there are two monomials $m_1$, $m_2$ in $\Xi_{g_2}$ such that 
$m=m_1m_2$. In the first case, by Equation \eqref{lw} we obtain that 
$\lambda_k \cdot [i_0,i_1,i_2,1] \geq 0$.  The second case follows because  if 
 $w_k(m_1m_2) \geq c_1(k)$, then $\lambda_k \cdot x_3^2x_0m_i \geq 0$. 
The same argument applies for monomials in $h_6$. Then, conditions \ref{wG} imply $\lambda_k \cdot  m \geq  0$.
\end{proof}
\begin{prop}\label{projdp}
An irreducible quintic surface $X$ with a curve of singularities of multiplicity two  is non-stable if and only if there is a coordinate system such that $F_X$ is given as Equation \eqref{qdp} and the branch locus associated to the morphism 
$
( G_{F_X}(\psi,x_0,x_1,x_2) = 0) \rightarrow \mathbb{P}^2
$
can be written as one of the following equations:
\begin{align*}
D_{\lambda_5}(x_0,x_1,x_2) 
&=
x_0^2
\left(
x_2^7x_0^3+\sum_{k=1}^{2}x_0^k\sum_{i=0}^{2}x_2^{3k-i}f_{10-4k+i}(x_0,x_1)
+f_{10}(x_0,x_1)
\right) 
\\
D_{\lambda_9}(x_0,x_1,x_2) &= x_2^6x_0^5x_1 +\sum_{i=0}^{5} x_2^ix_0^if_{12-2i}f(x_0,x_1).
\end{align*}
\end{prop}
\begin{proof}
The morphism  $(G_{X,p}=0) \to \mathbb{P}^2$ is generically finite of degree three. Its associated branch locus is given by the equation 
$4h_4(x_0,x_1,x_2)^3+27h_6(x_0,x_1,x_2)^2.$
By the results of Section \ref{sec:GIT}, a quintic surface, as in the statement, is non-stable if and only if there is a coordinate system where
 $\Xi_{F_X} \subset M^{\oplus}(\lambda_k)$ with $k \in \{ 5,9 \}$.   Then the polynomials $h_4$ and $h_6$ in the equation of the branch locus satisfy inequalities as described in Lemma \ref{dpcov}. Our statement describes the most general branch loci that satisfy those inequalities.
\end{proof}

\bibliographystyle{amsalpha}
\bibliography{quintic} 	
\end{document}